\documentclass[10pt]{article}
\usepackage{silence}
    \usepackage[utf8]{inputenc}
    \usepackage[T1]{fontenc}
    \usepackage{siunitx}
    \usepackage{colonequals}
    \usepackage[super]{nth}
    \usepackage{algorithm}
    \usepackage{amsmath}
    \usepackage{amsthm}
    \usepackage{amssymb}
    \usepackage{tikz}
    \usepackage{url}
    \usepackage{float}
    \usepackage{tabularx}
    \usepackage{pbox}
    \usepackage{tabulary}
    \usepackage{breqn}
    \usepackage{longtable}
    \usepackage{mdframed}
    \usepackage{booktabs}

    \usepackage{graphicx}
    \graphicspath{{./images/}{./}}

    \usepackage{wrapfig}
    \usepackage[font={small}]{caption}
    \usepackage{enumitem}
    \usepackage{physics}
    \usepackage{lipsum}
    \usepackage{mwe}
    \usepackage{breqn}
    \usepackage[noend]{algpseudocode}
    \usepackage{subcaption}
    \usepackage{graphicx}
    \usepackage{enumitem}
    \usepackage{bm}
    \usepackage{textcmds}

    \usepackage{apptools}
    \AtAppendix{\counterwithin{lemma}{section}}
    \AtAppendix{\counterwithin{remark}{section}}
    
    \usepackage{hyperref}
    \hypersetup{
        colorlinks,
            citecolor=black,
        filecolor=black,
        linkcolor=black,
        urlcolor=black
    }

\newtheorem{corollary}{Corollary}

\newtheorem{theorem}{Theorem}
\newtheorem{lemma}{Lemma}
\newtheorem{remark}{Remark}

\newtheorem{definition}{Definition}
\newtheorem{assumption}{Assumption}

\numberwithin{equation}{section} 
\numberwithin{lemma}{section} 
\numberwithin{theorem}{section} 
\numberwithin{definition}{section} 
\numberwithin{corollary}{section} 

\newcommand{\newL}{\tau_{\alpha}}

\newcommand{\fZeroMinusfStarOverH}{\frac{f(x_0) - f^*}{h(\epsilon, \alphaZero\gammaOne^{c+\newL}
)}}
\newcommand{\ntsAlphaZeroGammaOneCL}{N_{TS, \underline{\alphaZero \gammaOne^{c+\newL}}}}

\newcommand{\sqrtFrac}[2]{\sqrt{\frac{#1}{#2}}}
\newcommand{\gDeltaSDeltaOne}{g(\deltaS, \deltaOne)}

\newcommand{\mathO}[1]{\mathcal{O}\left( #1\right)}

\newcommand{\Z}{\mathbb{Z}} 
\newcommand{\N}{\mathbb{N}}
\newcommand{\R}{\mathbb{R}}
\newcommand{\E}{\mathbb{E}}
\renewcommand{\P}{\mathbb{P}}
\newcommand{\dotp}[2]{\langle #1, #2 \rangle}
\newcommand{\alphaLow}{\alpha_{low}}
\newcommand{\nEps}{N_{\epsilon}}

\newcommand{\nt}{N_T}
\newcommand{\nts}{N_{TS}}
\newcommand{\ns}{N_S}
\newcommand{\nuMe}{N_U}
\newcommand{\ntu}{N_{TU}}
\newcommand{\ntAlphaUpper}{N_{T, \overline{\alphaLowOne}}}
\newcommand{\nsAlphaUpper}{N_{S, \overline{\alphaLowOne}}}
\newcommand{\ntAlphaLower}{N_{T, \underline{\alphaLowOne}}}
\newcommand{\ntsAlphaLower}{N_{TS, \underline{\alphaLowOne}}}
\newcommand{\ntuAlphaLower}{N_{TU, \underline{\alphaLowOne}}}
\newcommand{\nuAlphaLower}{N_{U, \underline{\alphaLowOne}}}
\newcommand{\nsGammaAlpha}{N_{S, \underline{ \gamma_1^{c}\alphaLowOne}}}
\newcommand{\ntsGammaCAlpha}{N_{TS, \underline{ \gamma_1^c \alphaLowOne}}}
\newcommand{\nfsGammaAlpha}{N_{FS, \underline{ \gamma_1^c \alphaLowOne}}}
\newcommand{\nf}{N_F}

\newcommand{\logOneOverGammaOne}{\log(1/\gamma_1)}

\newcommand{\logGammaTwoOverGammaOne}{\log(\gamma_2/\gamma_1)}
\newcommand{\tCOne}{\newL}
\newcommand{\tCTwo}{c}

\newcommand{\expectation}[1]{\E \left[ #1 \right]}
\newcommand{\sumTk}{\sum_{k=0}^{N-1} T_k}
\newcommand{\eToMinusLambdaMinusOne}{e^{-\lambda}-1}
\newcommand{\oneMinusDeltaS}{1-\delta_S}
\newcommand{\conditionalE}[2]{\E \left[ #1 | #2 \right]}
\newcommand{\eToMinusLambdaSumTk}{e^{-\lambda \sumTk }}
\newcommand{\eToMinusLambdaSumTkNMinusTwo}{e^{-\lambda \sumTkToNMinusTwo }}
\newcommand{\exponentialMomentUpperTotal}{\left[ e^{( \eToMinusLambdaMinusOne)(\oneMinusDeltaS)}\right]^N}
\newcommand{\exponentialMomentUpper}{ e^{( \eToMinusLambdaMinusOne)(\oneMinusDeltaS)}}
\newcommand{\eToMinusLambda}[1]{e^{-\lambda #1}}
\newcommand{\conditionalP}[2]{\P \left[ #1 | #2 \right]}
\newcommand{\sumTkToNMinusTwo}{\sum_{k=0}^{N-2}T_k}
\newcommand{\SZeroDotsToXNMinusOne}{T_0, T_1, \dots, T_{N-2}, x_{N-1}}
\newcommand{\deltaOneComplicated}{-\delta_1 - (1-\delta_1) \log(1-\delta_1) }
\newcommand{\chernoffLowerExponential}{e^{-\frac{\delta_1^2}{2} (1-\delta_S) N}}

\newcommand{\nPreFactorTR}{\left[(1-\delta_S)(1-\delta_1) - 1 + \frac{c}{(c+1)^2} \right]^{-1}}
\newcommand{\betaK}{\beta_k}
\newcommand{\kStartOne}{k_{\small{start}}^{(1)}}
\newcommand{\kEndOne}{k_{\small{end}}^{(1)}}
\newcommand{\aOne}{A^{(1)}}
\newcommand{\closedInterval}[2]{\left[ #1, #2 \right]}
\newcommand{\twoCases}[4]{         \begin{cases}
             #1, & #2 \\
             #3, & #4.
        \end{cases} }
\newcommand{\closedOpenInterval}[2]{\left[ #1, #2 \right)}
\newcommand{\openClosedInterval}[2]{\left( #1, #2\right]}
\newcommand{\openInterval}[2]{\left( #1, #2\right)}
\newcommand{\mOneOne}{M_1^{(1)}}
\newcommand{\mTwoOne}{M_2^{(1)}}
\newcommand{\nOneOne}{n_1^{(1)}}
\newcommand{\nTwoOne}{n_2^{(1)}}
\newcommand{\kInStartEndOne}{k \in \openInterval{\kStartOne}{\kEndOne}}
\newcommand{\bOne}{B^{(1)}}
\newcommand{\kStartTwo}{k_{\small{start}}^{(2)}}
\newcommand{\kEndTwo}{k_{\small{end}}^{(2)}}
\newcommand{\kStartThree}{k_{\small{start}}^{(3)}}
\newcommand{\kStartI}{k_{\small{start}}^{(i)}}
\newcommand{\kStartIPlusOne}{k_{\small{start}}^{(i+1)}}
\newcommand{\kEndI}{k_{\small{end}}^{(i)}}
\newcommand{\aTwo}{A^{(2)}}
\newcommand{\mOneTwo}{M_1^{(2)}}
\newcommand{\mTwoTwo}{M_2^{(2)}}
\newcommand{\nOneTwo}{n_1^{(2)}}
\newcommand{\nTwoTwo}{n_2^{(2)}}

\newcommand{\bTwo}{B^{(2)}}

\newcommand{\kBar}{\overline{k}}
\newcommand{\onePlusTCTwo}{1+ \tCTwo}

\newcommand{\sHat}{\hat{s}}
\newcommand{\mKHat}[1]{\hat{m}_k\left( #1\right)}
\newcommand{\xK}{x_k}
\newcommand{\xKPlusOne}{x_{k+1}}
\newcommand{\sK}{s_k}
\newcommand{\normTwo}[1]{\left\lVert#1\right\rVert_2}
\newcommand{\gradFK}{\grad f(\xK)}
\newcommand{\DOneOverNMinusDTwo}{\frac{D_1}{N-D_2}}
\newcommand{\qInverseOfDOneOverNMinusDTwo}{q^{-1} \left( \DOneOverNMinusDTwo\right)}
\newcommand{\minGradient}[1]{\min_{k\leq #1} \normTwo{\gradFK}}
\newcommand{\minGradientN}{\minGradient{N}}

\newcommand{\DOneOverQPlusDTwo}{\frac{D_1}{q(\epsilon)}+D_2}

\newcommand{\logOneOverDelta}{\log\left( \frac{1}{\delta}\right)}
\newcommand{\eToMinusDThreeN}{e^{-D_3 N}}
\newcommand{\gammaOneC}{\gamma_1^c}
\newcommand{\hatNOneOne}{\hat{n}_1^{(1)}}
\newcommand{\hatNOneTwo}{\hat{n}_2^{(1)}}
\newcommand{\fK}{f(x_k)}
\newcommand{\innerProduct}[2]{\langle #1, #2 \rangle}
\newcommand{\sKGradFK}{S_k \gradFK}
\newcommand{\sKBKSKT}{S_k B_k S_k^T}
\newcommand{\rLTimesD}{\R^{l\times d}}
\newcommand{\fKPlusOne}{f(x_k + s_k)}
\newcommand{\epS}{\epsilon_S}
\newcommand{\sMax}{S_{\small{max}}}
\newcommand{\alphaMax}{\alpha_{\small{max}}}
\newcommand{\alphaK}{\alpha_k}
\newcommand{\hBar}{\bar{h}}
\newcommand{\deltaSOne}{\delta_S^{(1)}}
\newcommand{\deltaSTwo}{\delta_S^{(2)}}
\newcommand{\aKOne}{A_k^{(1)}}
\newcommand{\aKTwo}{A_k^{(2)}}
\newcommand{\intersect}{\cap}
\newcommand{\barXK}{\bar{x}_k}
\newcommand{\probabilityGivenXK}[1]{\probability{#1 | x_k = \barXK}}
\newcommand{\aKOneIntersectaKTwo}{\aKOne \intersect \aKTwo}
\renewcommand{\complement}[1]{\left( #1\right)^c}
\newcommand{\oneMiusEpsSToHalf}{\left( 1-\epS\right)^{1/2}}

\renewcommand{\abs}[1]{\lvert #1\rvert}

\newcommand{\alphaKPlusOne}{\alpha_{k+1}}
\newcommand{\gammaOne}{\gamma_1}
\newcommand{\gammaTwo}{\gamma_2}

\newcommand{\alphaZero}{\alpha_0}
\newcommand{\alphaMin}{\alpha_{\small{min}}}

\newcommand{\ceil}[1]{\left \lceil{#1}\right \rceil}
\newcommand{\logBaseGammaOne}[1]{\log_{\gammaOne}\left( #1\right)}
\newcommand{\minMe}[2]{\min \set{#1, #2}}
\newcommand{\alphaLowOne}{\alphaMin}
\newcommand{\pOneDef}{\frac{\log(\gammaTwo\alphaMin/\alphaZero)}{\logGammaTwoOverGammaOne}}

\newcommand{\gammaTwoExpression}{ \frac{1}{\gamma_1^c}}
\newcommand{\betaKPlusOne}{\beta_{k+1}}

\newcommand{\set}[1]{\left\{ #1 \right\}}

\newcommand{\probability}[1]{\P \left( #1 \right)}
\renewcommand{\complement}[1]{\left( #1 \right)^c}
\newcommand{\squareBracket}[1]{\left[ #1 \right]}

\newcommand{\bracket}[1]{\left( #1\right)}

\newcommand{\mkHatSHat}{\mKHat{\sHat}}
\newcommand{\oneOverTwoAlphaK}{\frac{1}{2\alphaK}}
\newcommand{\oneOverAlphaK}{\frac{1}{\alphaK}}
\newcommand{\SKTransposed}{S_k^T}

\newcommand{\BMax}{B_{max}}
\newcommand{\sKHat}{\hat{s}_k}
\newcommand{\kappaT}{\kappa_T}
\newcommand{\barH}[2]{\bar{h} \bracket{#1,#2}}

\newcommand{\halfBMax}{\frac{1}{2}\BMax}

\newcommand{\SKTransposedsKHat}{S_k^T \sKHat}
\newcommand{\lPlusHalfBmax}{L+\halfBMax}

\newcommand{\deltaOne}{\delta_1}

\newcommand{\eToQRySqaured}{e^{q \normTwo{Ry}^2}}
\newcommand{\eToQLOneMinusEps}{e^{ql(1-\epS)}}
\newcommand{\eToMinusQLOneMinusEps}{e^{-ql(1-\epS)}}
\newcommand{\texteq}[1]{\text{\quad #1}}

\newcommand{\aZeroOne}{A_0^{(1)}}
\newcommand{\aZeroTwo}{A_0^{(2)}}

\newcommand{\normInf}[1]{\|#1\|_{\infty}}

\newcommand{\deq}{\mathrel{\mathop:}=}

\newcommand{\logOneOverDeltaSTwo}{\log\bracket{1/\deltaSTwo}}
\newcommand{\mkHatSkHat}{\hat{m}_k\bracket{\hat{s}_k}}
\newcommand{\gK}{g_k}
\newcommand{\mK}[1]{m_k \bracket{#1}}
\newcommand{\deltaS}{\delta_S}
\newcommand{\integral}[3]{\int_{#1}^{#2} #3}
\newcommand{\probNEpsGrN}{\probability{\nEps >M}dM}

\newcommand{\CorTwoDeltaOne}{Suppose \eqref{eqn::deltaSConditionThmTwo} hold and let $\delta_1 \in (0,1)$ satisfy \eqref{eqn:tmp32}}
\newcommand{\IterKTrueandSuccssfulWithAlphaKGeqAlphaMin}{\substack{ \text{Iteration $k$ is true and successful} \\ \text{with $\alphaK \geq \alphaZero \gammaOne^{c+\newL}$}}}

\newcommand{\singleQuote}[1]{\lq #1\rq}

\usepackage[margin=1.in]{geometry} 
\numberwithin{equation}{section} 

\begin{document}

\title{Randomised subspace methods  for non-convex optimization, with applications to nonlinear least-squares}
\author{Coralia Cartis\thanks{Mathematical Institute, University of Oxford, Radcliffe Observatory Quarter, Woodstock Road, Oxford, OX2 6GG and the Alan Turing Institute for Data Science, British Library, NW1 2DB, London, UK; \texttt{cartis@maths.ox.ac.uk}. This author's work was supported by the Alan Turing Institute through the Turing Project scheme, by the Hong Kong Innovation and Technology Commission (InnoHK Project CIMDA) and was undertaken in collaboration with the Numerical Algorithms Group Ltd.}, 
Jaroslav Fowkes\thanks{Mathematical Institute, University of Oxford, Radcliffe Observatory Quarter, Woodstock Road, Oxford, OX2 6GG  and STFC Rutherford Appleton Laboratory, Chilton, Oxfordshire OX11 0QX, UK; 
\texttt{jaroslav.fowkes@stfc.ac.uk, @maths.ox.ac.uk}. This author's work was supported by the Oxford-Emirates Data Science Lab.}
\,\,and Zhen Shao\thanks{Mathematical Institute, University of Oxford, Radcliffe Observatory Quarter, Woodstock Road, Oxford, OX2 6GG, UK; \texttt{zhen.shao@maths.ox.ac.uk}. This author's work was supported by the EPSRC Centre For Doctoral Training in Industrially Focused Mathematical Modelling (EP/L015803/1) in collaboration with the Numerical Algorithms Group Ltd and by the EPSRC Doctoral Prize (EP/T517811/1(D4T00230)).}
}
\maketitle


 \begin{abstract}
We propose a general random subspace framework for unconstrained nonconvex optimization problems that requires a weak probabilistic assumption on the subspace gradient, which we show to be satisfied by various random matrix ensembles, such as Gaussian and sparse sketching, using Johnson-Lindenstrauss embedding properties. We show that, when safeguarded with trust region or quadratic regularization, this random subspace approach satisfies, with  high probability, a complexity bound of order $\mathcal{O}(\epsilon^{-2})$ to drive the (full) gradient below $\epsilon$; matching in the accuracy order, deterministic counterparts of these methods and securing almost sure convergence. Furthermore, no problem dimension dependence appears explicitly in the projection size of the sketching  matrix, allowing the choice of low-dimensional subspaces. We particularise this framework to Random Subspace Gauss-Newton (RS-GN) methods for nonlinear least squares problems, that only require the calculation of the Jacobian in the subspace; with similar complexity guarantees.  Numerical experiments with RS-GN on CUTEst nonlinear least squares are also presented, with some encouraging results. 

 \end{abstract}
\section{Introduction}
We investigate  the unconstrained optimization problem
\begin{equation}
    f^* = \min_{x \in \R^d} f(x),\label{eqn::fStar} 
\end{equation}
where $f:\mathbb{R}^d\rightarrow \mathbb{R}$ is a continuously differentiable and possibly nonconvex objective and where $d$ is  large. Considering the ever increasing scale of optimization problems, driven particularly by machine learning applications, we are interested in reducing the dimensionality of the parameter space by developing (random) subspace variants of classical algorithms. The two-fold advantages of such approaches are: the reduced cost of calculating problem derivatives as only their subspace projections  are needed; and of solving the subproblem given its much-reduced size.
Random block-coordinate descent methods are the simplest illustration of such an approach as they operate in a coordinate-aligned subspace, and only require a subset of partial derivatives to calculate an approximate gradient direction \cite{Nesterov12, richtarik2014iteration}. Despite documented failures of (deterministic) coordinate techniques on some problems \cite{Powell1973}, the challenges of large-scale applications have brought these methods to the forefront of research in the last decade; see 
\cite{Wright2015} for a survey of this active area, particularly for convex optimization.
Developments of these methods for nonconvex optimization can be found, for example, in \cite{Raza2013, Patrascu2015, lu2017randomized} and more recently, \cite{Yang2020}; with extensions to constrained problems \cite{Birgin1, Birgin2}, and distributed strategies \cite{Facchinei, facchinei2015parallel}.

Subspace methods can be seen as an extension of (block) coordinate methods by allowing the reduced variables
to vary in (possibly randomly chosen) subspaces that are not necessarily aligned with
coordinate directions. (Deterministic) subspace and decomposition methods have been of steadfast interest in the optimization community for decades. In particular, Krylov subspace methods can be applied to calculate -- at typically lower computational cost and in a matrix free way -- an approximate Newton-type search direction over increasing and nested subspaces; see \cite{Nocedal:2006uv}  for Newton-CG techniques (and further references) and \cite{GLTR, conn2000trust} for initial trust-region variants.  However, these methods still require that the full gradient vector is calculated/available at each iterations, as well as (full) Hessian matrix actions. The challenges of large scale calculations prompt us to go further, by imposing that only inexact gradient and Hessian action information is available, such as their projections to lower dimensional subspaces.
Thus, in these frameworks (which subsume block-coordinate methods), both the problem information and the search direction are inexact and low(er) dimensional. Deterministic proposals can be found in \cite{Yuan}, which also traces a history of these approaches. 

Random subspace approaches often rely on so-called {\it sketching} techniques \cite{10.1561/0400000060, 10.1561/2200000035}, in an attempt to exploit the benefits of
Johnson-Lindenstrauss (JL) Lemma-like results \cite{Johnson:1984aa}, 
that essentially reduce the dimension of the optimization problem 
without loss of information.  Both first-order and second-order variants have been proposed, particularly for convex and convex-like problems, that only calculate a (random) lower-dimensional projection of the gradient or Hessian. Sketched gradient descent methods have been proposed in 
\cite{Kozak_published, kozak2019stochastic, Gower2020, Grishchenko2021}.
 The sketched Newton algorithm \cite{pilanci2017newton} requires a sketching matrix 
 that is proportional to the rank of the Hessian, 
 which may be too computationally expensive if the Hessian has high rank.
 By contrast, sketched online Newton \cite{luo2016efficient}  
 uses streaming sketches to scale up a second-order method, 
 comparable to Gauss–Newton, for solving online learning problems.
The randomised subspace Newton \cite{gower2019rsn} efficiently sketches 
the full Newton direction for a family of generalised linear models, 
such as logistic regression.
The stochastic dual Newton ascent algorithm in  \cite{qu2016sdna} 
requires a positive definite upper bound $M$ on the Hessian and 
proceeds by selecting random principal submatrices of $M$ 
that are then used to form and solve an approximate Newton system. 
The randomized block cubic Newton method in \cite{doikov2018randomized} 
combines the ideas of randomized coordinate descent with 
cubic regularization and requires the optimization problem to be block separable. Concomitant with our work \cite{Zhen-PhD, zhen:icml_BCGN}, \cite{Gower2022} proposes a sketched Newton method for the solution of (square or rectangular) nonlinear systems of equations, with a global convergence guarantee.   Improvements to the solution of large-scale least squares problems are given in  \cite{lacotte2019faster, ergen2019, lacotte2020optimal, kahale2020leastsquares}. 
Random subspace methods have also been studied for the global optimization of nonconvex function, especially in the presence of low effective dimensionality of the objective, when the latter is only varying in a fixed but unknown low-dimensional subspace; see \cite{adilet3} and the references therein. This special structure assumption has also been investigated in the context of local optimization \cite{subramanian}, but it is beyond our scope here.  

Sketching can also be applied not only to reduce the dimension of the parameter/variable domain, but also the data/observations when minimizing an objective given as a  sum of (many smooth) functions, as it is common when training machine learning systems \cite{Curtis} or in data fitting/regression problems. Then, using sketching, we subsample some of these constituent functions and calculate a local improvement for this reduced objective; this leads to stochastic gradient-type methods and related variants, namely, stochastic algorithms\footnote{Note that in the case of random subspace methods, the objective is evaluated accurately as opposed to observational sketching, where this accuracy is lost through subsampling.}. A vast literature is available on this topic (see for example, \cite{Gower:2016up, Berahas:2019wr, doi:10.1287/ijoo.2019.0043, 10.5555/3305381.3305577, MR4163540, 2019arXiv190403342W, Roosta-Khorasani2019}) but not directly relevant to the developments in this paper.


A more general random framework that allows inexact gradient, Hessian and even function values involves {\it probabilistic models} \cite{MR3245880}, where the local model at each iteration is only sufficiently accurate with a certain probability. Such local models can be derivative-based \cite{Cartis:2017fa, doi:10.1287/ijoo.2019.0016, Gratton:2017kz} or derivative-free \cite{MR3245880, doi:10.1287/ijoo.2019.0016, MR3800867, Gratton:2017kz, 2019arXiv190403342W}, constructed from (exact or inexact) function evaluations only. Our results fit into this framework but use significantly milder assumptions on the model construction than in probabilistic models. There have been several follow ups and other approaches in the context of derivative-free optimization, see \cite{Lindon22} and references therein for a detailed survey\footnote{Since our focus here is on derivative-based methods, we are not surveying in detail the derivative-free optimization advances and direct the reader to \cite{Lindon22}.}. 
The results we present here, due to their generality and mild assumptions, have already been applied (using earlier drafts of this manuscript and  \cite{zhen:icml_BCGN, Zhen-PhD}) to derivative-free optimization, such as in \cite{Lindon22}.




\paragraph{Summary of contributions} The theoretical contributions of our paper are two fold\footnote{A brief description, without proofs, of a subset of the results of this paper has appeared as part of a four-page conference proceedings paper (without any supplementary materials) in the ICML Workshop “Beyond first order methods in ML systems” (2020), see \cite{zhen:icml_BCGN}. We also note that a substantial part of this paper has been included as Chapter 4 of the doctoral thesis \cite{Zhen-PhD}.} Firstly, we extend the
derivative-based probabilistic models algorithmic framework \cite{Cartis:2017fa, Gratton:2017kz} to allow a more diverse set of algorithm parameter choices and to obtain a 
high probability global rate of convergence of the form \eqref{eqn:tmp33}
 (rather than an almost-sure convergence result in expectation), which is a more precise complexity result and needed for our subsequent developments. Then, within this framework -- but under much weaker assumptions than for standard probabilistic models (see Remark \ref{Prob_Model_remark}) -- we develop a generic random-subspace framework based on sketching techniques (see Definition \ref{def::true_iters}). 
The latter conditions are similar to (some of the conditions in) \cite{Kozak_published, kozak2019stochastic} (and were discovered independently of the latter works); however, our framework is more general, aims to solve nonconvex problems, and allows several algorithm variants (first- and second-order, adaptive) and sketching matrices to be used. 

In particular, using Johnson-Lindenstrauss (JL) embedding properties of the random matrices employed to sketch/construct the projected random subspace,  we show that random subspace methods, with trust region or quadratic regularization strategies (or other),
    have a global worst-case complexity of order $\mathO{\epsilon^{-2}}$ to drive the gradient $\nabla f(x_k)$ of $f$ below the desired accuracy $\epsilon$, with exponentially high probability; this complexity bound matches in  the order of the accuracy that of corresponding deterministic/full dimensional variants of these same methods. The choice of the random subspace only needs to project approximatively correctly the length of the full gradient vector; a mild requirement that can be achieved by several sketching matrices such as scaled Gaussian matrices, and some sparse embeddings. In these cases, the same embedding properties provide that the dimension of the projected subspace is {\it independent} of the ambient dimension and so the algorithm can operate in a small dimensional subspace at each iteration, making the projected step and gradient much less expensive to compute. The choice of the random ensemble may bring some dimension dependence in the bound, which may be eliminated if $d$ is proportional to $l$, the size of the sketching subspace.  
    We also show that in the case of sampling sketching matrices, when our approach reduces to  randomised block-coordinates,
    the success of the algorithm on non-convex smooth problems depends on the non-uniformity of the gradient; an intuitive connection that captures the fact that if the gradient has some components that are significantly more important than others, and if these components are missed by the uniform sampling strategy, then too much information is lost and convergence may fail. Thus almost sure convergence of randomized block-coordinate methods can be secured under additional problem assumptions (which are not needed in the case of Gaussian or other JL-embedding matrices).
    
    We particularize our general sketching framework to global safeguarding strategies such as trust region and quadratic regularization that ensure (almost sure) convergence from any starting point;  
    as well as to local models that use approximate second-order information as in the case of Gauss-Newton type methods for nonlinear least squares. Our random-subspace Gauss-Newton methods for 
nonlinear least-squares problems 
only need a sketch of the Jacobian matrix in the variable domain at each iteration, 
which it then uses to solve a reduced linear least-squares problem for 
the step calculation.  Finally, we illustrate our theoretical findings numerically, using random-subspace Gauss-Newton variants with 
three sketching matrix ensembles,
 on some CUTEst subproblems.

The structure of the paper is as follows. Section 2 presents a variant of the algorithmic framework of probabilistic models, and extends its theory to obtain a high probability complexity bound under very general assumptions. Section 3 particularizes this framework to the case when the probabilistic  local model is calculated in a random subspace and provides general conditions under which such a framework converges almost surely and with proven complexity bound. Section 4 adds the remaining ingredients (namely, quadratic regularization and trust region) for devising complete algorithms, complete with ensuing convergence guarantees. Finally, Section 6 considers nonlinear least squares problems and further specifies the sketched local model in this case by using a Gauss-Newton model with a sketched Jacobian matrix; numerical results are also presented.

\section{A general algorithmic framework with random models}

\subsection{A generic algorithmic framework and some assumptions}

We first describe a generic algorithmic framework that is similar to \cite{Cartis:2017fa} and that encompasses the
main components of the unconstrained optimization schemes we analyse
in this paper. Some of the key assumptions required in our analysis also resemble the set up in \cite{Cartis:2017fa}. 
Despite these similarities, our analysis and results are different and qualitatively improve upon those in \cite{Cartis:2017fa}\footnote{This is in the sense  that, for example, \autoref{thm2} implies the main result [Theorem 2.1] in \cite{Cartis:2017fa}.}. 

The scheme relies on building a local, reduced model of the objective
function at each iteration, 
minimizing this model or reducing it in a sufficient manner and
considering the step which is dependent on a stepsize parameter and
which provides the model reduction (the stepsize parameter may be
present in the model or independent of it). 
This step determines a new candidate point. The function value is then
computed (accurately) at the new candidate point. 
If the function reduction provided by the candidate point is deemed
sufficient, then the iteration is declared successful, the 
candidate point becomes the new iterate and the step size parameter is
increased. Otherwise, the iteration is 
unsuccessful, the iterate is not updated and the step size parameter is reduced.

We summarize the main steps of the generic framework below\footnote{Throughout the paper, we let $\N^+=\N\setminus \{0\}$ denote the set of positive natural numbers.}.

\begin{algorithm}[H]
\begin{description}
\item[Initialization] \ \\
Choose a class of  (possibly random) models $m_k\left(w_k(\sHat)\right) = \mKHat{\sHat}$, where $\sHat \in \R^l$ with $l\leq d$ is the step parameter and $w_k: \R^l \to \R^d$ is the prolongation function. 
Choose constants $\gamma_1\in (0,1)$, $\gamma_2 = \gammaOne^{-c}$, for some $c \in \N^+$,  $\theta \in (0,1)$ and $\alpha_{\max}>0$.
Initialize the algorithm by setting $x_0 \in \R^d$, $\alpha_0 = \alphaMax \gamma_1^p$ for some $p \in \N^+$ and $k=0$.

 \item[1. Compute a reduced model and a step] \ \\
Compute a local (possibly random) reduced model $\mKHat{\sHat}$ of $f$ around $x_k$ with $\mKHat{0} = f(x_k)$. \\
Compute a step parameter $\sHat_k(\alpha_k)$, where the parameter $\alpha_k$ is present in the reduced model or the step parameter computation.\\
Compute a potential step $s_k = w_k(\sHat_k)$.

\item[2. Check sufficient decrease]\ \\  
Compute $f(x_k + s_k)$ and check if sufficient decrease (parameterized by $\theta$) is achieved in $f$ with respect to $\mKHat{0} - \mKHat{\hat{\sK}(\alpha_k)}$.

\item[3, Update the parameter $\alphaK$ and possibly take the trial step $\sK$]\ \\
If sufficient decrease is achieved, set $\xKPlusOne = \xK + \sK$ and $\alphaKPlusOne = \min \set{\alphaMax, \gammaTwo\alphaK}$ [successful iteration].
Otherwise set $\xKPlusOne = \xK$ and $\alphaKPlusOne = \gammaOne \alphaK$ [unsuccessful iteration].\\
Increase the iteration count by setting $k=k+1$ in both cases. 

\caption{\bf{Generic optimization framework based on  randomly generated reduced models.}} \label{alg:generic} 

\end{description}
\end{algorithm}

We extend the framework in \cite{Cartis:2017fa} so that the proportionality constants for increasing/decreasing the step size parameter are not required to be strictly reciprocal, but may differ up to an integer power (see \autoref{AA3}). Another difference is that though not explicitly stated, the local model $\mK{x}$ in \cite{Cartis:2017fa} seems to assume that the variable $x$ has the same dimension as the parameter of the objective function $f$. This can be seen in the definition of true iterations in \cite{Cartis:2017fa} which assumes that the gradient of the model has the same dimension as the gradient of the function $f$, as well as in all the main cases/examples given there.
 By contrast, our framework \autoref{alg:generic} explicitly states that the model does not need to have the same dimension as the objective function; with the two being connected by a step transformation function $w_k: \R^l \to \R^d$ which typically here will have $l<d$.

As an example, note that letting $l=d$ and $w_k$ be the identity function in \autoref{alg:generic} leads to usual, full dimensional local models 
which coupled with typical strategies of linesearch and trust-region as parametrised by $\alpha_k$ or regularization  (given by $1/\alpha_k$),
recover classical, deterministic variants of corresponding methods; see \cite{Cartis:2017fa} for more details.

Since the local model is (possibly) random, $\xK, \sK, \alphaK$ are in general random variables; we will use $\barXK, \bar{s}_k, \bar{\alpha}_k$ to denote their realizations.  Given  (any) $\epsilon >0$, we define convergence in terms of the random variable
\begin{equation}
    \nEps:= \inf \{M: \min_{k\leq M} \|\grad f(x_k)\|_2 \leq \epsilon \}, \label{eqn::nEps}
\end{equation}
which represents the first time that the (true/unseen) gradient descends below $\epsilon$. Note that this is the same definition as in \cite{Cartis:2017fa} and could equivalently be defined as $\nEps = \min \set{k: \normTwo{\gradFK}\leq \epsilon}$; we use the above choice for reasons of generality. Also note that $k< \nEps$ implies $\normTwo{\grad f(x_k)} > \epsilon$, which will be used repeatedly in our proofs.




Let us suppose that there is a subset of iterations, which we refer to as \textbf{true iterations} such that \autoref{alg:generic} satisfies the following assumptions. The first assumption states that given the current iterate, an iteration $k$ is true at least with a fixed probability, and is independent of the truth value of all past iterations.

\begin{assumption}\label{AA2}
There exists $\delta_S \in (0,1)$ such that for any $\barXK \in \R^d$ and $k=1, 2, \dots$
\begin{equation}
    \probabilityGivenXK{T_k} \geq 1-\delta_S, \notag
\end{equation}
where $T_k$ is defined as
        \begin{equation}
            T_k = 
        \twoCases{1}{\text{if iteration $k$ is true}}{0}{\text{otherwise}}
        \label{eqn::t_k}
        \end{equation}

Moreover, $\probability{T_0} \geq 1-\delta_S$; and $T_k$ is conditionally independent of $T_0, T_1, \dots, T_{k-1}$ given $x_k = \barXK$. 
\end{assumption}

The next assumption says that for $\alphaK$ small enough, any true iteration before convergence is guaranteed to be successful.

\begin{assumption}\label{AA3}
For any $\epsilon >0$, there exists an iteration-independent constant $\alphaLow>0$ (that may depend on $\epsilon$ as well as problem and algorithm parameters) such that 
if iteration $k$ is true, $k< \nEps$, and $\alpha_k \leq \alphaLow$, then iteration $k$ is successful. 

\end{assumption}

The next assumption says that before convergence, true and successful iterations result in an objective decrease bounded below by an (iteration-independent) function $h$, which is monotonically increasing in its two arguments, $\epsilon$ and $\alphaK$. 

\begin{assumption}\label{AA4}
There exists a non-negative, non-decreasing function $h(z_1,z_2)$ such that,
for any $\epsilon>0$, 
if iteration $k$ is true and successful with $k < \nEps$, then 

	\begin{equation}
	f(\xK) - f(\xK + \sK) \geq h(\epsilon, \alpha_k), \label{eqn::generic_model_decrease}
	\end{equation}
where $s_k$ is computed in Step 1 of \autoref{alg:generic}. Moreover, $h(z_1, z_2)>0$ if both $z_1>0$ and $z_2>0$.
\end{assumption}

The final assumption requires that the function values are monotonically decreasing throughout the algorithm. 

\begin{assumption}\label{AA5}
For any $k \in \N$, we have 
\begin{equation}
    f(\xK) \geq f(\xKPlusOne). \label{eqn::fKNonIncreasing}
\end{equation}
\end{assumption}

The following Lemma is a simple consequence of \autoref{AA3}. 

\begin{lemma}
\label{lem::alphaMin}
Let $\epsilon>0$ and \autoref{AA3} hold with $\alphaLow>0$. Then there exists $\newL \in \N^+$, and $\alphaMin >0$ such that 
\begin{align}
&\newL = \ceil{\logBaseGammaOne{ \minMe{\frac{\alphaLow}{\alphaZero}}{\frac{1}{\gammaTwo}}}}, \label{eqn:tauLDef}\\
&\alphaMin = \alphaZero \gammaOne^\newL, \label{eqn::alphaMin}\\
    &\alphaMin \leq \alphaLow, \notag \\
    &\alphaMin \leq \frac{\alphaZero}{\gammaTwo}, \label{eqn::alphaMinUpperByGammaTwoOverAlphaZero}
\end{align}
where $\gammaOne, \gammaTwo, \alphaZero$ are defined in \autoref{alg:generic}. 
\end{lemma}
\begin{proof}
We have that $\alphaMin \leq \alphaZero \gammaOne^{\logBaseGammaOne{ \frac{\alphaLow}{\alphaZero}}} = \alphaLow$. Therefore by \autoref{AA3}, if iteration $k$ is true, $k< \nEps$, and $\alpha_k \leq \alphaMin$ then iteration $k$ is successful. Moreover, $\alphaMin \leq \alphaZero \gammaOne^{\logBaseGammaOne{\frac{1}{\gammaTwo}}} = \frac{\alphaZero}{\gammaTwo} = \alphaZero \gammaOneC$. It follows from $\alphaMin = \alphaZero \gammaOne^\newL$ that $\newL \geq c$. Since $c \in \N^+$, we have $\newL \in \N^+$ as well.

\end{proof}

\subsection{A probabilistic convergence result}

\autoref{thm2} is our main result concerning the convergence of  \autoref{alg:generic}. It states a probabilistic bound on the total number of iterations $\nEps$ required by the generic framework to converge to within  $\epsilon$-accuracy of first-order optimality.

\begin{theorem}
\label{thm2}
Let \autoref{AA2}, \autoref{AA3}, \autoref{AA4} and \autoref{AA5} hold with $\epsilon>0$, $\delta_S \in (0,1)$, $\alphaLow>0$, $h: \R^2 \to \R$ and $\alphaMin = \alphaZero \gammaOne^\newL$ associated with $\alphaLow$, for some $\newL \in \N^+$; assume also that 
\begin{equation}
    \deltaS < \frac{c}{(c+1)^2}, \label{eqn::deltaSConditionThmTwo}
\end{equation}
where $c$ is chosen at the start of \autoref{alg:generic}. Suppose that \autoref{alg:generic} runs for $N$ iterations\footnote{For the sake of clarity, we stress that $N$ is a deterministic constant, namely, the total number of iterations that we run \autoref{alg:generic}. $\nEps$, the number of iterations needed before convergence, is a random variable.}.
Then, for any $\delta_1 \in (0,1)$ such that 
\begin{equation}
    \gDeltaSDeltaOne >0, \label{eqn:tmp32}
\end{equation}
where 
\begin{equation}
    g(\deltaS, \deltaOne) = \nPreFactorTR, \label{eqn:gDeltaSDeltaOneDef}
\end{equation}
if $N$ satisfies 
\begin{equation}
    N \geq \gDeltaSDeltaOne \squareBracket{
         \fZeroMinusfStarOverH
         + \frac{\newL}{1+c}}, \label{eqn::n_upper_2}
\end{equation}
we have that
\begin{equation}
    \probability{N \geq \nEps} \geq 1 - \chernoffLowerExponential. \label{eqn:tmp33}
\end{equation}
\end{theorem}

The proof of \autoref{thm2} is delegated to the Appendix. 

\begin{remark}
Note that $c/(c+1)^2 \in (0, 1/4]$ for $c \in \N^+$, and so \eqref{eqn::deltaSConditionThmTwo} and \eqref{eqn:tmp32} can only be satisfied for some $c$ and $\deltaOne$ given that $\deltaS < \frac{1}{4}$. Thus our theory requires that an iteration is true with probability at least $\frac{3}{4}$. Compared to the analysis in \cite{Cartis:2017fa}, which requires that an iteration is true with probability at least $\frac{1}{2}$, our condition imposes a stronger requirement. This is due to the high probability nature of our result, while the convergence result in \cite{Cartis:2017fa} is in expectation. Furthermore, we will see in \autoref{lem::deduceAA2} that we are able to impose arbitrarily small value of $\delta_S$, thus satisfying this requirement, by choosing an appropriate dimension of the local reduced model $\mKHat{\sHat}$.
\end{remark}

\begin{remark} We illustrate how our result leads to Theorem 2.1 in \cite{Cartis:2017fa}, which concerns $\expectation{\nEps}$. We have, with $N_0$ defined as the right hand side of \eqref{eqn::n_upper_2}, 
\begin{align}
    \expectation{\nEps} &= \integral{0}{\infty}{\probNEpsGrN} \notag \\
    &= \integral{0}{N_0}{\probNEpsGrN} + \integral{N_0}{\infty}{\probNEpsGrN} \notag \\
    & \leq N_0 + \integral{N_0}{\infty}{\probNEpsGrN} \notag \\
    & \leq N_0 + \integral{N_0}{\infty}{e^{-\frac{\delta_1^2}{2} (1-\delta_S) M} dM} \notag \\
    & = N_0 + \frac{2}{\delta_1^2 (1-\deltaS)} e^{-\frac{\delta_1^2}{2} (1-\delta_S) N_0}, \notag
\end{align}
where we used \autoref{thm2} to derive the last inequality. The result in \cite{Cartis:2017fa} is of the form  $\expectation{\nEps} \leq N_0$. Note that the difference term between the two bounds is exponentially small in terms of $N_0$ and therefore our result is asymptotically the same as that in \cite{Cartis:2017fa}.
\end{remark}

\subsection{Consequences of \autoref{thm2}}
We state and prove three corollaries of \autoref{thm2}, under mild assumptions on $h$. These results illustrate two different aspects of \autoref{thm2}.

The following expressions will be used, as well as $\gDeltaSDeltaOne$ defined in \eqref{eqn:tmp32},
\begin{align}
    &q(\epsilon) = h(\epsilon, \gammaOneC\alphaMin), \label{eqn:qEps} \\
    &D_1 = \gDeltaSDeltaOne (f(x_0) - f^*), \label{eqn::D_1} \\
    &D_2 = \gDeltaSDeltaOne \frac{\newL}{1+c}, \label{eqn::D_2}\\
    &D_3 = \frac{\delta_1^2}{2}(1-\delta_S). \label{eqn::D_3}
\end{align}
From \eqref{eqn:qEps}, \eqref{eqn::D_1}, \eqref{eqn::D_2}, \eqref{eqn::D_3}, a sufficient condition for \eqref{eqn::n_upper_2} to hold is
\begin{align}
    N &\geq \gDeltaSDeltaOne \squareBracket{\frac{f(x_0)-f^*}{h(\epsilon, \gammaOneC\alphaMin)} + \frac{\tCOne}{1+\tCTwo}} \notag \\
    & = \frac{D_1}{q(\epsilon)} + D_2; \notag
\end{align}
and \eqref{eqn:tmp33} can be restated as 
\begin{equation}
    \probability{N > \nEps} \geq 1 - e^{-D_3 N}. \notag
\end{equation}

The first corollary gives the rate of change of $\minGradientN$ as $N \to \infty$. It will yield a rate of convergence by substituting in a specific expression of $h$ (and hence $q^{-1}$).
\begin{corollary}
\label{Cor_1}
Let \autoref{AA2}, \autoref{AA3}, \autoref{AA4}, \autoref{AA5} hold. Let $f^*, q, D_1, D_2, D_3$ be defined in \eqref{eqn::fStar}, \eqref{eqn:qEps}, \eqref{eqn::D_1}, \eqref{eqn::D_2} and \eqref{eqn::D_3}. \CorTwoDeltaOne. 
Then for any $N \in \N$ such that $\qInverseOfDOneOverNMinusDTwo$ is well-defined, we have
\begin{equation}
    \probability{\minGradient{N} \leq \qInverseOfDOneOverNMinusDTwo} \geq 1 - e^{-D_3 N}. \label{eqn:tmp34}
\end{equation}
\end{corollary}

\begin{proof}
Let $N \in \N$ such that $\qInverseOfDOneOverNMinusDTwo$ exists and let $\epsilon = \qInverseOfDOneOverNMinusDTwo$. Then we have
\begin{equation}
    \probability{\minGradientN \leq \qInverseOfDOneOverNMinusDTwo}
    = \probability{\minGradientN \leq \epsilon}
    \geq \probability{N > \nEps}, \label{eq::tmp9}
\end{equation}
where the inequality follows from the fact that $N \geq \nEps$ implies $\minGradientN \leq \epsilon$. 
On the other hand, we have
    $N = D_1/\frac{D_1}{N-D_2}+ D_2 
    = \frac{D_1}{q(\epsilon)}+D_2$.
Therefore \eqref{eqn::n_upper_2} holds; and applying \autoref{thm2}, we have that $\probability{N \geq \nEps} \geq 1 - e^{-D_3 N}$. Hence \eqref{eq::tmp9} gives the desired result. 
\end{proof}

The next Corollary restates \autoref{thm2} for a fixed, arbitrarily-high, success probability. 
\begin{corollary}
\label{Cor_2}
Let \autoref{AA2}, \autoref{AA3}, \autoref{AA4}, \autoref{AA5} hold.
\CorTwoDeltaOne. Then for any $\delta \in (0,1)$, suppose
\begin{equation}
    N \geq \max \set{\DOneOverQPlusDTwo, \frac{\logOneOverDelta}{D_3}}, \label{eqn::tmp10}
\end{equation}
where $D_1, D_2, D_3, q$ are defined in \eqref{eqn::D_1}, \eqref{eqn::D_2}, \eqref{eqn::D_3} and \eqref{eqn:qEps}. Then
\begin{equation}
    \probability{\minGradientN < \epsilon} \geq 1-\delta.  \notag
\end{equation}

\end{corollary}

\begin{proof}
We have 
\begin{equation}
    \probability{\minGradientN \leq \epsilon}
    \geq \probability{N \geq \nEps}
    \geq 1- e^{-D_3 N}
    \geq 1-\delta, \notag
\end{equation}
where the first inequality follows from definition of $\nEps$ in \eqref{eqn::nEps}, the second inequality follows from \autoref{thm2} (note that \eqref{eqn::tmp10} implies \eqref{eqn::n_upper_2}) and the last inequality follows from \eqref{eqn::tmp10}.
\end{proof}

The next Corollary gives the rate of change of the expected value of $\minGradientN$ as $N$ increases.

\begin{corollary}
\label{Cor_3}
Let \autoref{AA2}, \autoref{AA3}, \autoref{AA4}, \autoref{AA5} hold.
\CorTwoDeltaOne. Then for any $N \in \N$ such that $\qInverseOfDOneOverNMinusDTwo$ exists, where $q, D_1, D_2$ are defined in \eqref{eqn:qEps}, \eqref{eqn::D_1}, \eqref{eqn::D_2}, we have
\begin{equation}
    \expectation{\minGradientN} \leq \qInverseOfDOneOverNMinusDTwo + \normTwo{\grad f(x_0)} \eToMinusDThreeN, \notag
\end{equation}
where $D_3$ is defined in \eqref{eqn::D_3} and $x_0$ is chosen in \autoref{alg:generic}.
\end{corollary}

\begin{proof}
We have
\begin{equation}
    \begin{split}
    & \expectation{\minGradientN} \notag \\
     &\leq \probability{\minGradientN \leq \qInverseOfDOneOverNMinusDTwo} \qInverseOfDOneOverNMinusDTwo \notag \\
    &+ \probability{\minGradientN > \qInverseOfDOneOverNMinusDTwo}\normTwo{\grad f(x_0)} \notag \\ 
    & \leq \qInverseOfDOneOverNMinusDTwo +  \eToMinusDThreeN \normTwo{\grad f(x_0)},  \notag
    \end{split}
\end{equation}
where to obtain the first inequality, we split the integral in the definition of expectation
\begin{align*}
 & \expectation{ \minGradientN}
 = \integral{0}{\infty}{\probability{\minGradientN = x} x dx} \\
 & = \integral{0}{\qInverseOfDOneOverNMinusDTwo}{\probability{\minGradientN = x} x dx} + \integral{\qInverseOfDOneOverNMinusDTwo}{\infty}{\probability{\minGradientN = x} x dx},
\end{align*}
and used $\probability{\minGradientN = x} = 0$ for $x> \normTwo{\grad f(x_0)}$ which in turn, follows from $\minGradientN \leq \normTwo{\grad f(x_0)}$. For the second inequality,  
we used \eqref{eqn:tmp34}, $\probability{\minGradientN \leq \qInverseOfDOneOverNMinusDTwo}\leq 1$ and  $$\probability{\minGradientN >\qInverseOfDOneOverNMinusDTwo} \leq \eToMinusDThreeN$$.
\end{proof}

\section{A general algorithmic framework with random subspace models}
The aim of this section is to particularise the generic algorithmic framework (\autoref{alg:generic}) and its analysis to the special case when the random local models
are generated by random projections, and thus lie in a lower dimensional random subspace. 

\subsection{A generic random-subspace method using sketching}
\autoref{alg:sketching} particularises \autoref{alg:generic} by specifying the local model as one that lies in a lower dimensional subspace  generated by sketching using a random matrix. We also define the step transformation function and the criterion for sufficient decrease. 
The details of the step computation and the adaptive step parameter are deferred to the next section, where complete algorithms will be given.

\begin{algorithm}[H]
\begin{description}

 \item[Initialization] \ \\
 Choose a matrix distribution $\cal{S}$ of matrices $S\in \rLTimesD$. Let $\gamma_1, \gamma_2, \theta, \alphaMax, x_0, \alpha_0$ be defined in \autoref{alg:generic} with $\mKHat{\sHat}$ and $w_k$ specified below in \eqref{eqn::mKHatSpec} and \eqref{eqn::wKSpec}. 

 \item[1. Compute a reduced model and a step] \ \\
 In Step 1 of \autoref{alg:generic}, draw a random matrix $S_k \in \R^{l \times d}$ from $\cal{S}$, and let
 \begin{align}
     &\mKHat{\sHat} = \fK + \innerProduct{\sKGradFK}{\sHat} + \frac{1}{2} \innerProduct{\sHat}{\sKBKSKT \sHat}; \label{eqn::mKHatSpec} \\
     &w_k(\sHat_k) = S_k^T \sHat_k, \label{eqn::wKSpec}
 \end{align}
 where $B_k \in \R^{d\times d}$ is a user-chosen matrix. 
 
 Compute $\sKHat$ by approximately minimising $\mKHat{\sHat}$, for $\sHat \in \mathbb{R}^l$, such that  $\mKHat{\sKHat} \leq \mKHat{0}$ where $\alphaK$ is the (same) algorithm parameter as in \autoref{alg:generic},
 and set $\sK = w_k(\sKHat)$ as in \autoref{alg:generic}.

\item[2. Check sufficient decrease]\ \\  
In Step 2 of \autoref{alg:generic}, let sufficient decrease be defined by the condition
\begin{equation}
    \fK - \fKPlusOne \geq \theta \squareBracket{\mKHat{0} - \mKHat{\hat{\sK}(\alpha_k)}}. \label{eqn::sufDecreaseSpec}
\end{equation}

\item[3. Update the parameter $\alphaK$ and possibly take the trial step $\sK$]\ \\
Follow Step 3 of \autoref{alg:generic}.

\caption{\bf{A generic random-subspace method using sketching}} \label{alg:sketching} 
\end{description}
\end{algorithm}

Clearly, for the choice of $S_k$, the case of interest in \autoref{alg:sketching} is when $l\ll d$, so that the local model is low dimensional. 
A full-dimensional (deterministic) local model would be typically chosen as
$$m_k(s)= \fK + \innerProduct{\nabla f(x_k)}{s} + \frac{1}{2} \innerProduct{s}{B_ks}, \quad s\in \mathbb{R}^d,$$
in standard nonlinear optimization algorithms such as linesearch, trust region and regularization methods \cite{Nocedal:2006uv}, for some approximate Hessian matrix $B_k$ (that could also be absent).
Letting $s=S_k^T\hat{s}$ in $m_k(s)$, and using adjoint/transposition properties,  yield our reduced model $\mKHat{\sHat}$ in  \autoref{alg:sketching} where $\sKGradFK$ and $S_k^TB_kS_k$ are now the reduced/projected/subspace gradient and approximate Hessian, respectively. The advantages of such a reduced local model is that it needs only to be minimized (approximately) 
over $\mathbb{R}^l$ with $l<d$; and that only a reduced/projected/approximate gradient is needed to obtain an approximate first-order model, thus potentially or in some cases, reducing the computational cost of obtaining problem information, which is a crucial aspect of efficient large-scale optimization.

Using the  criterion for sufficient decrease, we have that \autoref{AA5} is satisfied by \autoref{alg:sketching}.

\begin{lemma} \label{lem:AA8impliesAA5}
\autoref{alg:sketching} satisfies \autoref{AA5}.
\end{lemma}
\begin{proof}
If iteration $k$ is successful, \eqref{eqn::sufDecreaseSpec} with $\theta \geq 0$ and $\mKHat{\sKHat} \leq \mKHat{0}$ (specified in \autoref{alg:sketching}) give $\fK-\fKPlusOne \geq 0$. If iteration $k$ is unsuccessful, we have $s_k=0$ and therefore $\fK - \fKPlusOne = 0$. 
\end{proof}

Next, we define the true iterations  for \autoref{alg:sketching} and show \autoref{AA2} is satisfied when $\cal{S}$ is a variety of random ensembles. 


\begin{definition} \label{def::true_iters}
Iteration $k$ is a true iteration if
\begin{align}
    &\normTwo{\sKGradFK}^2 \geq (1- \epS)\normTwo{\gradFK}^2, \label{eqn::JL} \\
    & \normTwo{S_k}\leq \sMax, \label{eqn::sMax}
\end{align}
where $S_k \in \R^{l \times d}$ is the random matrix drawn in Step 1 of \autoref{alg:sketching}, and $\epS \in (0,1), \sMax>0$ are iteration-independent constants.
\end{definition}

We note that Definition \ref{def::true_iters} is to the best of our knowledge, the weakest general requirement on the quality of the approximate gradient information that ensures almost sure convergence of such a general framework. In particular, it is milder than probabilistically fully-linear model conditions that require componentwise agreement between the gradient and its approximation (within some adaptive tolerance), with some probability. 

\begin{remark}\label{Prob_Model_remark}
In \cite{Cartis:2017fa, MR3245880, MR3800867}, true iterations are required to satisfy (with some probability)
\begin{equation}
    \normTwo{\grad m_k(0) - \gradFK} \leq \kappa \alphaK \normTwo{\grad m_k(0)}, \notag
\end{equation}
where $\kappa>0$ is a constant and $\alphaK$ in their algorithm is bounded by $\alphaMax$. The above equation implies
\begin{equation}
    \normTwo{\grad m_k(0) } \geq \frac{\normTwo{\gradFK}}{1+\kappa \alphaMax}, \notag
\end{equation}
which implies \eqref{eqn::JL} with $1-\epS = \displaystyle\frac{1}{1 + \kappa\alphaMax}$ (with some probability). 
Thus our requirement on the quality of the problem information is milder than the one in the above papers using probabilistic models (and further confirmed in the next section by the variety of random ensembles satisfying our definition). 
\end{remark}

Using Definition \ref{def::true_iters} of the true iterations, \autoref{AA2} holds if the following two conditions on the random matrix distribution $\cal{S}$ are met. 

\begin{assumption} \label{AA6}
There exists $\epS, \deltaSOne \in (0,1)$ such that for a(ny) fixed
$y \in \set{\grad f(x): x \in \R^d}$, $S_k$ drawn from $\cal{S}$ satisfies
\begin{equation}
    \probability{\normTwo{S_k y}^2 \geq (1-\epS)\normTwo{y}^2} 
    \geq 1-\deltaSOne. \label{eqn:tmp36}
\end{equation}
\end{assumption}

\begin{assumption} \label{AA7}
There exists $\deltaSTwo\in [0,1), \sMax>0$ such that for $S_k$ randomly 
drawn from $\cal{S}$, we have
\begin{equation}
    \probability{\normTwo{S_k} \leq \sMax} \geq 1-\deltaSTwo. \notag
\end{equation}
\end{assumption}

\begin{lemma} \label{lem::deduceAA2}
Let \autoref{AA6} and \autoref{AA7} hold with $\epS,\deltaSTwo \in (0,1), \deltaSOne\in [0,1), \sMax >0$. Suppose that $\deltaSOne+\deltaSTwo<1$. Let true iterations be defined in \autoref{def::true_iters}. Then \autoref{alg:sketching} satisfies \autoref{AA2} with $\delta_S = \deltaSOne + \deltaSTwo$.
\end{lemma}




\begin{proof}[Proof of \autoref{lem::deduceAA2}]
Let $\barXK \in \R^d$ be given, which determines $\grad f(\barXK) \in \R^d$. Let $\aKOne$ be the event that \eqref{eqn::JL} holds and $\aKTwo$ be the event that \eqref{eqn::sMax} holds. Thus $T_k = \aKOne \intersect \aKTwo$. Note that given $x_k = \barXK$, $T_k$ only depends on $S_k$, which is independent of all previous iterations. Hence $T_k$ is conditionally independent of $T_0, T_1 \dots, T_{k-1}$ given $x_k = \barXK$. 
Next, using Boole's inequality in probability, we have for $k \geq 1$, 
\begin{equation}
    \probabilityGivenXK{\aKOne \intersect \aKTwo } \geq 1 - \probabilityGivenXK{\complement{\aKOne}} - \probabilityGivenXK{\complement{\aKTwo}}. \label{eqn::tmp17}
\end{equation}
We also obtain the following,
\begin{align}
    \probabilityGivenXK{\aKOne} &= \conditionalP{\aKOne}{x_k = \barXK, \gradFK = \grad f(\barXK)} 
     = \conditionalP{\aKOne}{\gradFK = \grad f(\barXK)} 
     \geq 1-\deltaSOne,\label{eqn::tmp18}
\end{align}
where the first equality follows from the fact that $\xK = \barXK$ implies $\gradFK = \grad f(\barXK)$; the second equality follows from the fact that given $\gradFK = \grad f(\barXK)$, $\aKOne$ is independent of $\xK$; and the inequality follows from applying \autoref{AA6} with $y = \grad f(\barXK)$.
On the other hand, as $\aKTwo$ is independent of $x_k$, we have that
\begin{equation}
    \probabilityGivenXK{\aKTwo} = \probability{\aKTwo} \geq 1-\deltaSTwo,\label{eqn::tmp19}
\end{equation}
where the inequality follows from \autoref{AA7}.
It follows from \eqref{eqn::tmp17} using \eqref{eqn::tmp18} and \eqref{eqn::tmp19} that for $k\geq 1$, 
\begin{equation}
    \probabilityGivenXK{\aKOneIntersectaKTwo} \geq 1-\deltaSOne - \deltaSTwo = 1-\delta_S. \notag
\end{equation}
For $k=0$, we have $\probability{\aZeroOne} \geq 1-\deltaSOne$ by \autoref{AA6} with $y=\grad f(x_0)$ and $ \probability{\aZeroTwo} \geq 1-\deltaSTwo $ by \autoref{AA7}. So $\probability{\aZeroOne \intersect \aZeroTwo} \geq 1-\delta_S$ by Boole's inequality.
\end{proof}


\subsection{Some suitable choices of sketching matrices}\label{BCGN:randomMatrixDistr}

Next, we detail some random matrix distributions $\cal{S}$ and associated quantities that can be used in \autoref{alg:sketching} and that satisfy 
\autoref{AA6} and \autoref{AA7}. Random matrix theory \cite{10.1561/0400000060, MR3837109, Gupta1999} and particularly, Johnson-Lindenstrauss lemma-type \cite{Johnson:1984aa}  results will prove crucial. 

\subsubsection{Gaussian sketching matrices}

(Scaled) Gaussian matrices have independent and identically distributed normal entries \cite{Gupta1999}. 
\begin{definition}\label{def:Gaussian}
$S=(S_{ij}) \in \R^{ l \times d }$ is a scaled Gaussian matrix if its entries $S_{ij}$ are independently distributed as $N (0, {l}^{-1})$.
\end{definition}
The next result, which is a consequence of the scaled Gaussian matrices being an oblivious Johnson-Lindenstrauss embedding \cite{10.1561/0400000060}, shows that using such  matrices with \autoref{alg:sketching} satisfies \autoref{AA6}. The proof is included for completeness in the appendix, but can also be found in \cite{MR1943859}. 

\begin{lemma} \label{lem:GaussJLEmbedding}
Let $S\in \R^{l \times d}$ be a scaled Gaussian matrix so that each entry is $N(0, l^{-1})$. Then $S$ satisfies \autoref{AA6} for any $\epS \in (0,1) $ and $\deltaSOne = e^{-\epS^2 l /4}$. 
\end{lemma}

The following bound on the maximal singular value of a scaled Gaussian matrix is needed in our pursuit of satisfying \autoref{AA7}.




\begin{lemma}[Theorem 2.13 in \cite{MR1863696}] \label{lem:Davidson}
Given $l,d \in \N$ with $l \leq d$, consider the $d\times l$ matrix $\Gamma$ whose entries are independently distributed as $N(0, {d}^{-1})$. Then for any $\delta >0$,\footnote{We set $t = \sqrt{\frac{2 \logOneOverDelta}{d} }$ in the original statement of this theorem.}
\begin{equation}
    \probability{\sigma_{max} \bracket{\Gamma} \geq 1 + \sqrt{\frac{l}{d}} + \sqrt{\frac{2 \log (1/\delta)}{l} }} < \delta, \label{eqn:Davidson_upper_Gaussian}
\end{equation}
where $\sigma_{max}(.)$ denotes the largest singular value of its matrix argument.
\end{lemma}

The next lemma shows that  \autoref{AA7} is satisfied by scaled Gaussian matrices.
\begin{lemma}\label{Lem:GaussSMax}
Let $S \in \R^{l\times d}$ be a scaled Gaussian matrix. Then $S$ satisfies \autoref{AA7} for any $\deltaSTwo \in (0,1)$ and 
\begin{equation}
    \sMax = 1 + \sqrt{\frac{d}{l}} + \sqrt{\frac{2\logOneOverDeltaSTwo}{l}}. \notag
\end{equation}
\end{lemma}

\begin{proof}
We have $\normTwo{S} = \normTwo{S^T} = \sqrt{\frac{d}{l}} \normTwo{\sqrt{\frac{l}{d}}S^T}$. Applying \autoref{lem:Davidson} with $\Gamma = \sqrt{\frac{l}{d}}S^T$, we have that 
\begin{equation*}
    \probability{\sigma_{max} \bracket{ \sqrt{\frac{l}{d}}S^T} \geq 1 + \sqrt{\frac{l}{d}} + \sqrt{\frac{2 \logOneOverDeltaSTwo}{d}} }
    < \deltaSTwo.
\end{equation*}
Noting that $\normTwo{S} = \sqrtFrac{d}{l} \sigma_{max} \bracket{\Gamma}$, and taking the event complement gives the result.
\end{proof}


Unfortunately, Gaussian matrices are dense and thus computationally expensive to use algorithmically; sparse ensembles are much better as we shall see next.

\subsubsection{Sparse sketching: \texorpdfstring{$s$-hashing}{TEXT} matrices}
Comparing to Gaussian matrices, $s$-hashing matrices, including in the case when $s=1$,
are sparse, having $s$ nonzero entries per column, and  they preserve the sparsity (if any) of the vector/matrix they act on; 
and the corresponding linear algebra  is computationally faster. 
\begin{definition} \cite{10.1561/0400000060}  \label{def:s-hashing}
We define $S \in \R^{l \times d}$  to be an $s$-hashing matrix if, 
independently for each $j \in [d]$, we sample without replacement 
$i_1, i_2, \dots, i_s \in [l]$ uniformly at random and 
let $S_{i_k j} = \pm 1/\sqrt{s}$, $k = 1, 2, \dots, s$.
\end{definition}

The next two lemmas show that $s$-hashing matrices satisfy \autoref{AA6} and \autoref{AA7}.



\begin{lemma}[Theorem 13 in \cite{MR3167920}, and also Theorem 5 in \cite{MR3773205}\footnote{The latter reference gives a simpler proof.}] \label{thm:s_hashing}
Let $S \in \rLTimesD$ be an $s$-hashing matrix. 
Then $S$ satisfies \autoref{AA6} for any $\epS \in (0,1)$ 
and 
$\deltaSOne = e^{-\frac{l\epS^2}{C_1}}$ provided that $s = C_2 \epS l$,
where $C_1, C_2$ are problem-independent constants. 
\end{lemma}

\begin{lemma}
Let $S \in \rLTimesD$ be an $s$-hashing matrix. Then $S$ satisfies \autoref{AA7} with $\deltaSTwo = 0$
and $\sMax = \sqrtFrac{d}{s}$.
\end{lemma}
\begin{proof}
Note that for any matrix $A\in\rLTimesD$, $\normTwo{A} \leq \sqrt{d} \normInf{A}$; and $\normInf{S} = \frac{1}{\sqrt{s}}$. The result follows by combining these two facts. 
\end{proof}



\subsubsection{Sparse sketching: (Stable) $1$-hashing matrices}
In \cite{CHEN2020105639}, a variant of $1$-hashing matrix is proposed that satisfies \autoref{AA6} with an improved bound  $\sMax$. Its construction is  as follows.

\begin{definition}\label{def::stable-1-hashing}
Let $l<d \in \N^+$. A stable $1$-hashing matrix $S \in \R^{l \times d}$ has one non-zero per column, whose value is $\pm 1$ with equal probability, with the row indices of the non-zeros given by the sequence $\cal{I}$ constructed as follows.  Repeat $[l]$ (that is, the set $\set{1,2, \dots, l}$) for $\ceil{d/l}$ times to obtain a set $D$. Then randomly sample $d$ elements from $D$ without replacement to construct the sequence $\cal{I}$.\footnote{We may also conceptually think of $S$ as being constructed by taking the first $d$ columns of a random column permutation of the matrix $T = \squareBracket{I_{l\times l}, I_{l \times l}, \dots, I_{l \times l}}$ where the identity matrix $I_{l \times l}$ is concatenated by columns $\ceil{d/l}$ times.}
\end{definition}
\begin{remark}
Comparing to a $1$-hashing matrix, a stable $1$-hashing matrix still has $1$ non-zero per column. However its construction guarantees that each row has at most $\ceil{d/l}$ non-zeros because the set $D$ has at most $\ceil{d/l}$ repeated row indices and the sampling is done without replacement. 
\end{remark}

In order to satisfy  \autoref{AA6}, we need to following result.
\begin{lemma}[Theorem 5.3 in \cite{CHEN2020105639}] \label{tmp-2021-12-31-3}
Let $S \in \R^{l\times d}$ be given in \autoref{def::stable-1-hashing}.
Then, for $0< \epsilon, \delta < 1/2$, there exists 
$l = \mathO{\frac{\log(1/\delta)}{\epsilon^2}}$ 
such that for any $x\in \R^d$, we have that 
$$\probability{ \|Sx\|_2 \geq (1 - \epsilon) \|x\|_2} > 1-\delta.$$
\end{lemma}

\begin{lemma}\label{lem:tmp:2022-1-13-1}
Let $S \in \R^{l \times d}$ be a stable $1$-hashing matrix. Let $\epS \in (0,3/4)$, $C_3$, a problem-independent constant, and suppose that $e^{-\frac{l(\epS-1/4)^2}{C_3}} \in (0, 1/2)$. Then $S$ satisfies \autoref{AA6} with $
\deltaSOne = e^{-\frac{l(\epS-1/4)^2}{C_3}}$. 
\end{lemma}
\begin{proof}
Let $\bar{\epsilon} = \epS - 1/4 \in (0, 1/2)$.
From \autoref{tmp-2021-12-31-3}, we have that there exists $C_3>0$ such that for $\deltaSOne = e^{-\frac{l(\epS-1/4)^2}{C_3}}$, $S$ satisfies
$\probability{ \|Sx\|_2 \geq (1 - \bar{\epsilon}) \|x\|_2} > 1-\deltaSOne$. 
Note that $\|Sx\|_2 \geq (1-\bar{\epsilon}) \|x\|_2$ implies 
$\|Sx\|_2^2 \geq (1 - 2\bar{\epsilon} + \bar{\epsilon}^2) \|x\|_2$. Thus 
$\|Sx\|_2^2 \geq (1-\bar{\epsilon}-1/4)\|x\|_2^2$ as $\bar{\epsilon}^2-\bar{\epsilon} \geq -1/4$ for $\bar{\epsilon} \in (0,1/2)$. The desired result follows.
\end{proof}



The next lemma shows that using stable $1$-hashing matrices satisfies \autoref{AA7}. Note that the bound $\sMax$ is smaller than that for $1$-hashing matrices; and, assuming $l>s$, smaller than that for $s$-hashing matrices as well. 
\begin{lemma}\label{lem:stable-1-hashing-SMax}
Let $S \in \R^{l\times d}$ be a stable $1$-hashing matrix. Then $S$ satisfies \autoref{AA7} with $\deltaSTwo=0$ and $\sMax = \sqrt{\ceil{d/l}}$.
\end{lemma}
\begin{proof}
Let $D$ be defined in \autoref{def::stable-1-hashing}. We have that
\begin{align}
\normTwo{Sx}^2 
&= (\sum_{1\leq j \leq d, {\cal{I}}(j)=1} \pm x_j)^2 
+ (\sum_{1\leq j \leq d, {\cal{I}}(j)=2} \pm x_j)^2
+ \dots + (\sum_{1\leq j \leq d, {\cal{I}}(j)=l} \pm x_j)^2 \\
& \leq 
(\sum_{1\leq j \leq d, {\cal{I}}(j)=1}  |x_j|)^2 
+ (\sum_{1\leq j \leq d, {\cal{I}}(j)=2} |x_j|)^2
+ \dots + (\sum_{1\leq j \leq d, {\cal{I}}(j)=l} |x_j|)^2 \\
&\leq 
\ceil{d/l} \bracket{
\sum_{1\leq j \leq d, {\cal{I}}(j)=1}  x_j^2
+ \sum_{1\leq j \leq d, {\cal{I}}(j)=2}  x_j^2
+ \dots + \sum_{1\leq j \leq d, {\cal{I}}(j)=l}  x_j^2
}  \\
& = \ceil{d/l} \|x\|_2,
\end{align}
where the $\pm$ on the first line results from the non-zero entries of $S$ having random signs, and the last inequality follows since for any vector $v \in \R^n$, 
$\|v\|_1^2\leq  n \|v\|_2^2$; and ${\cal{I}}(j) = k$  for at most $\ceil{d/l}$ indices $j$.
\end{proof}

\subsubsection{Sparse sketching: sampling matrices}\label{sampling_mat_paragraph}

(Scaled) sampling matrices $S\in\R^{l\times d}$ randomly select entries/rows of the vector/matrix it acts on (and scale it). 
\begin{definition}\label{def:sampling}
We define $S=(S_{ij}) \in \R^{l \times d}$ to be a scaled sampling matrix if, independently for each $i \in [l]$, we sample $j \in [d]$ uniformly at random and let $S_{ij}=\sqrt{\frac{d}{l}}$. 
\end{definition}

Next we show that sampling matrices satisfy \autoref{AA6}. The following expression that represents the maximum non-uniformity of the objective gradient will be used,
\begin{equation}
    \nu = \max \set{\frac{\normInf{y}}{\normTwo{y}}, y=\grad f(x) \text{ for some } x\in \R^d}. \label{eq:nu_def}
\end{equation}

The following concentration result will be useful.

\begin{lemma}[\cite{MR2946459}]\label{lem:Tropp_matrix_chernoff}
Consider a finite sequence of independent random numbers $\set{X_k}$ that 
satisfies $X_k \geq 0$ and $\abs{X_k} \leq P$ almost surely. 
Let $\mu = \sum_k \expectation{X_k}  $, then $\probability{\sum_k X_k \leq (1-\epsilon) \mu} \leq 
e^{-\frac{\epsilon^2 \mu}{2P}}$. 
\end{lemma}

\begin{lemma} \label{lem:sampling:non-uniformity-BCGN}
Let $S \in \rLTimesD$ be a scaled sampling matrix and
$\nu$ given in \eqref{eq:nu_def}.
Then $S$ satisfies \autoref{AA6} for any $\epS \in (0,1)$,
with $\deltaSOne = e^{- \frac{\epS^2 l}{2d\nu^2}}$.
\end{lemma}

\begin{proof}
Note that \eqref{eqn:tmp36} is invariant to the scaling of $y$
and trivial for $y=0$. Therefore we may assume $\normTwo{y}=1$ 
without loss of generality. 
We have $\normTwo{Sy} = \frac{l}{d} \sum_{k=1}^l
\squareBracket{\bracket{Ry}_k}^2 $,
where $R \in \rLTimesD$ is an (un-scaled) sampling matrix \footnote{Namely, each row of $R$ has a $1$ in a random column.} and $(Ry)_k$ denotes 
the $k^{th}$ entry of $Ry$.
Let $X_k = \squareBracket{\bracket{Ry}_k}^2$.
Note that because the rows of $R$ are independent, 
the $X_k$ variables are also independent. 
Moreover, as $\bracket{Ry}_k$ equals some entry of $y$,
and $\normInf{y}\leq \nu$ by definition of $\nu$ and 
$\normTwo{y} = 1$, we have $\squareBracket{\bracket{Ry}_k}^2 \leq \nu^2$.
Finally, note that $\expectation{X_k} = \frac{1}{d}\normTwo{y}^2 = 
\frac{1}{d}$, so that $\sum_k \expectation{X_k} = \frac{l}{d}$.
Applying \autoref{lem:Tropp_matrix_chernoff} 
with $\epsilon = \epS$ we have
\begin{equation}
    \probability{\sum_{k=1}^l \squareBracket{\bracket{Ry}_k}^2
    \leq (1-\epS) \frac{l}{d}} \geq e^{- \frac{\epS^2 l}{2d\nu^2}}. \notag
\end{equation}
Using $\normTwo{Sy}^2 =\frac{l}{d} \sum_{k=1}^l 
\squareBracket{\bracket{Ry}_k}^2$ gives the result. 
\end{proof}


We note that this theoretical property of scaled sampling matrices is different from the corresponding one for Gaussian/$s$-hashing matrices in the sense that the required value of $l$ now depends on $\nu$. Note that $\frac{1}{d} \leq \nu^2 \leq 1$ (with both bounds tight). Therefore in the worst case, for fixed value of $\epS, \deltaSOne$, $l$ is required to be $\mathO{d}$ and no dimensionality reduction can be achieved using sketching. This is not surprising given that sampling-based random methods often require adaptively increasing the sampling size for convergence. However, for \singleQuote{nice} objective functions such that $\nu^2 = \mathO{\frac{1}{d}}$, sampling matrices have similar theoretical properties as Gaussian/$s$-hashing matrices. The appeal of using sampling lies in the fact that only a subset of the entries of the gradient need to be evaluated when calculating $S_k\nabla f(x_k)$ in \autoref{alg:sketching}.


Sampling matrices also have bounded Euclidean norms, so that \autoref{AA7} is satisfied.  
\begin{lemma}\label{tmp-2022-1-14-12}
Let $S \in \rLTimesD$ be a scaled sampling matrix. Then \autoref{AA7} is satisfied with $\deltaSTwo=0$ and $\sMax=\sqrt{\frac{d}{l}}$.
\end{lemma}
\begin{proof}
We have that $\normTwo{Sx}^2 \leq \frac{d}{l}\normTwo{x}^2$ for any $x\in \R^d$.
\end{proof}

\subsubsection{Summary of sketching results} \label{summary_sketching_paragraph}

We summarise the sketching results in this subsection in \autoref{tab:alg:sketching}, where we also give the sketching dimension $l$ in terms of $\epS$ and $\deltaSOne$ by rearranging the expressions for $\deltaSOne$. Note that for $s$-hashing matrices, $s$ is required to be $C_2 \epS l$ (see \autoref{thm:s_hashing}), while for scaled sampling matrices, $\nu$ is defined in \eqref{eq:nu_def}. 
Furthermore, note that the sketching accuracy  $\epS \in (0,1)$ (that is different than the optimality accuracy $\epsilon$ that \autoref{alg:sketching} secures probabilistically  in the gradient size) need not be small; in fact, $\epS=\mathO{1}$.
The potentially large error in the gradient sketching that is allowed by the algorithm is due to its iterative nature, that mitigates the inaccuracies of the embedding; as illustrated by the complexity bound in \autoref{thm:complexity-QR-Gaussian}. 

\begin{table}[h]
\small
\begin{tabular}{|l|l|l|l|l|l|}
\hline
                                         & $\epS$                              & $\deltaSOne$                                      & $l$          & $\deltaSTwo$ & $\sMax$              \\ \hline
Scaled Gaussian  & $(0,1)$               & $e^{-\frac{\epS^2 l}{4}}$           & $4\epS^{-2} \log(\frac{1}{\deltaSOne})$           & $(0,1)$      &  $1 + \sqrt{\frac{d}{l}} + \sqrt{\frac{2\logOneOverDeltaSTwo}{l}}$ \\ \hline
$s$-hashing         & $(0,1)$            & $e^{-\frac{\epS^2 l}{C_1}}$         & $C_1 \epS^{-2} \log(\frac{1}{\deltaSOne})$        & $0$          & $\sqrtFrac{d}{s}$    \\ \hline
Stable $1$-hashing  & $(0,\frac{3}{4})$ & $e^{-\frac{l(\epS-1/4)^2}{C_3}}$     & $C_3 (\epS-1/4)^{-2} \log(\frac{1}{\deltaSOne}) $ & $0$          & $\sqrt{\ceil{\frac{d}{l}}}$  \\ \hline
Scaled sampling    & $(0,1)$            & $e^{- \frac{\epS^2 l}{2d\nu^2}}$ & $2d\nu^2 \epS^{-2} \log(\frac{1}{\deltaSOne})$    & $0$          & $\sqrt{\frac{d}{l}}$ \\ \hline
\end{tabular}
\caption{Summary of theoretical properties of different random ensembles to be used within \autoref{alg:sketching}.}
\label{tab:alg:sketching}
\end{table}

Other random ensembles are possible, for example, Subsampled Randomised Hadamard Transform, Hashed Randomised Hadamard Transform, (which have the effect of allowing vectors with smaller $\nu$ to be sketched correctly, or allow smaller values of $s$ (such as $s=1$) in the choice of $s$-hashing $S_k$) and many more \cite{2021arXiv210511815C, 10.1561/0400000060, Zhen-PhD}.

\section{Random subspace quadratic regularization and trust region algorithms}
In this section, we further particularise our general framework to concrete and complete algorithms, with associated complexity results.
Namely, we specify the algorithm parameter $\alpha_k$ in  \autoref{alg:sketching} in two different ways, and its role in the 
reduced model, leading to a random-subspace quadratic regularization variant and a random-subspace trust region one, respectively, both with iteration complexity of  $\mathO{\epsilon^{-2}}$ to bring the objective's gradient below $\epsilon$, with high probability; for a diverse set of sketching matrices. This complexity is derived straightforwardly from the general results in the previous section by showing that the two variants satisfy the remaining assumptions, namely \autoref{AA3} and \autoref{AA4}. We will also include conditions that allow approximate calculation of the reduced step $\hat{s}$. 

We note that linesearch variants of our framework are also straightforwardly possible, where the parameter $\alpha_k$ is now the linesearch/stepsize parameter (chosen for example to satisfy an Armijo condition), and $B_k$ is either set to the zero or to some positive definite matrix on each iteration; their complexity  can be derived very similarly to the below; see for example, \cite{Cartis:2017fa}, for more general probabilistic linesearch variants.
 


The following results are needed for both algorithmic variants that we describe, hence we present them in a slightly more general way.
Namely, we show that \autoref{alg:sketching} satisfies \autoref{AA4} if the following model reduction condition is met.

\begin{assumption} \label{AA8}
There exists a non-negative, non-decreasing function $\hBar: \R^2 \to \R$ such that on each true iteration $k$ of \autoref{alg:sketching} we have
\begin{equation}
    \mKHat{0} - \mKHat{\sHat_k(\alpha_k)} \geq \hBar\bracket{\normTwo{\sKGradFK}, \alphaK}, \notag
\end{equation}
where $S_k, \hat{m}_k, \alphaK, \sHat_k$ are defined in \autoref{alg:sketching}.
\end{assumption}

\begin{lemma} \label{lem:AA8impliesAA4}
Let \autoref{AA8} hold with $\hBar$ and true iterations defined in \autoref{def::true_iters}. Then \autoref{alg:sketching} satisfies \autoref{AA4} with $h(\epsilon, \alphaK) = \theta \hBar \bracket{ (1-\epS)^{1/2} \epsilon, \alphaK}$, where $\epS$ is defined in \eqref{eqn::JL}.
\end{lemma}

\begin{proof}
Let $k$ be a true and successful iteration with $k < \nEps$ for some $\epsilon>0$ where $\nEps$ is defined in \eqref{eqn::nEps}. Then, using the fact that the iteration is true, successful, \autoref{AA8} and $k<\nEps$, we have
\begin{align}
    \fK-\fKPlusOne 
    & \geq \theta \squareBracket{\mKHat{0} - \mKHat{\sHat_k(\alphaK)}}
    \geq \theta \hBar( \normTwo{\sKGradFK}, \alphaK) \notag\\
    & \geq \theta \hBar( \oneMiusEpsSToHalf \normTwo{\gradFK}, \alphaK) 
     \geq \theta \hBar( \oneMiusEpsSToHalf \epsilon, \alphaK). \notag
\end{align}
\vspace*{-0.2cm}
\end{proof}

The next Lemma is a standard result and we include its proof for completeness in Appendix \ref{BCGN:aux-app}. It will be needed later on, to show our random-subspace variants satisfy \autoref{AA3}.

\begin{lemma} \label{lem:Taylor}
Assume that the objective function $f$ in problem \eqref{eqn::fStar} is continuously differentiable with $L$-Lipschitz continuous gradient. Let
  \autoref{alg:sketching} be applied to \eqref{eqn::fStar}, where the choice of $B_k$ is such that $\normTwo{B_k} \leq \BMax$ for all $k$, and some constant $\BMax\geq 0$. Then 
\begin{equation}
    | f(x_k + s_k) - \mkHatSkHat| \leq \bracket{\frac{L+\BMax}{2}} \normTwo{S_k^T \hat{\sK}}^2, \label{eqn::Alg2Lipschitz}
\end{equation}
where $\sKHat \in \R^l$, $S_k\in \R^{l \times d}$ and $s_k = S_k^T \sKHat \in \R^d$.
\end{lemma}


\subsection{A random-subspace quadratic regularisation algorithm}
Here we present \autoref{alg:sketching_QR}, a random subspace quadratic regularisation method that uses sketching, which is a particular form of \autoref{alg:sketching} where the step is computed using a quadratic regularisation approach. We state the algorithm in a self-contained way, and then identify its similarities to \autoref{alg:sketching}.

\begin{algorithm}[H]
\begin{description}
 \item[Initialization] \ \\
 Choose a matrix distribution $\cal{S}$ of matrices $S \in \R^{l\times d}$. 
 Choose constants $\gamma_1\in (0,1)$, $\gamma_2 = \gammaOne^{-c}$, for some $c \in \N^+$, $l \in \N^+$, $\theta \in (0,1)$, 
 $\alpha_{\max}>0$ and $\kappaT\geq 0$.
  Initialize the algorithm by setting $x_0 \in \R^d$, $\alpha_0 = \alphaMax \gamma_1^p$ for some $p \in \N^+$ and $k=0$.

 \item[1. Compute a reduced model and a step] \ \\
 Draw a random matrix $S_k \in \R^{l \times d}$ from $\cal{S}$, and let
 \begin{align}
     \mkHatSHat = \fK + \innerProduct{\sKGradFK}{\sHat} + \frac{1}{2} \innerProduct{\sHat}{\sKBKSKT \sHat}
 \end{align}
 where $B_k \in \R^{d\times d}$ is a positive semi-definite user-chosen matrix (the choice $B_k=0$ is allowed).
 
 Compute $\sKHat$ by approximately minimising $\displaystyle\hat{q}_k(\hat{s}) = \mkHatSHat + \frac{1}{2\alphaK}\normTwo{S_k^T\sK}^2$ 
 such that the following two conditions hold
\begin{align}
    \normTwo{\grad \hat{q}_k(\hat{s}_k)} \leq \kappaT \normTwo{\SKTransposed \sKHat}, \label{eqn:QRskHatCond1}\\[1ex]
   \hat{q}_k(\hat{s})  \leq \hat{q}_k(0), \label{eqn:QRsKHatCond2}
\end{align}
 and\footnote{Note that $\hat{q}_k(0)=\hat{m}_k(0)=f(x_k)$.} set $\sK = S_k^T \sKHat$.

\item[2. Check sufficient decrease]\ \\  
Check the sufficient decrease condition 
\begin{equation}
    \fK - \fKPlusOne \geq \theta \squareBracket{\mKHat{0} - \mKHat{\hat{\sK}}}. \notag
\end{equation}

\item[3. Update the parameter $\alphaK$ and possibly take the trial step $\sK$]\ \\
If sufficient decrease is achieved, set $\xKPlusOne = \xK + \sK$ and $\alphaKPlusOne = \min \set{\alphaMax, \gammaTwo\alphaK}$ [successful]. \\
Otherwise set $\xKPlusOne = \xK$ and $\alphaKPlusOne = \gammaOne \alphaK$ [unsuccessful].\\
Increase the iteration count by setting $k=k+1$ in both cases.

\caption{\bf{A random-subspace quadratic regularisation algorithm using sketching}} \label{alg:sketching_QR} 
\end{description}
\end{algorithm}


 \autoref{alg:sketching_QR} is identical to \autoref{alg:sketching} apart from the details of the introduction of the regularized reduced model $\hat{q}_k$ and associated step calculation in the second part of Step 1; this also potentially requires the introduction of a 
user-chosen parameter $\kappaT$ that allows approximate subproblem solution.

We note that
\begin{equation}
    \mKHat{0} - \mKHat{\sKHat} =  \hat{q}_k(0)- \hat{q}_k(\hat{s}_k)+ \oneOverTwoAlphaK \normTwo{\SKTransposed\sKHat}^2 \geq \oneOverTwoAlphaK \normTwo{\SKTransposed\sKHat}^2\geq 0, \label{eqn:QRModelDecreaseLowerByStep}
\end{equation}
where we have used \eqref{eqn:QRsKHatCond2}; this implies that condition $\mKHat{0}\geq \mKHat{\sKHat}$ required in \autoref{alg:sketching} is implicitly achieved. Furthermore, 
\autoref{lem:QRAlpahLow} shows \autoref{alg:sketching_QR} satisfies \autoref{AA3}. 

\begin{lemma} \label{lem:QRAlpahLow}
Assume that the objective function $f$ in problem \eqref{eqn::fStar} is continuously differentiable with $L$-Lipschitz continuous gradient. Let
  \autoref{alg:sketching_QR} be applied to \eqref{eqn::fStar}, where the choice of $B_k$ is such that $\normTwo{B_k} \leq \BMax$ for all $k$, and some constant $\BMax\geq 0$. Then 
   \autoref{alg:sketching_QR} satisfies \autoref{AA3} with 
\begin{equation}
    \alphaLow = \frac{1-\theta}{L + \BMax}. \notag
\end{equation}
\end{lemma}

\begin{proof}
Let $\epsilon>0$ and $k < \nEps$, and assume iteration $k$ is true with $\alphaK \leq \alphaLow$. Let
  $\rho_k = [f(x_k) - f(x_k + s_k)]/[\mKHat{0} - \mKHat{\sKHat}]$, which together with $\mKHat{0} = \fK$, implies
   $ \abs{1- \rho_k} 
    = \abs{f(\xK + \sK) - \mKHat{\sK}}/\abs{\mKHat{0}-\mKHat{\sKHat}}$. Thus 
    $$\abs{1- \rho_k} \leq \frac{\bracket{\frac{L+\BMax}{2}} \normTwo{\SKTransposedsKHat}^2}{\frac{1}{2\alpha_k} \normTwo{\SKTransposedsKHat}^2}
     \leq 1-\theta,$$ 
where the first inequality follows from \autoref{lem:Taylor} and \eqref{eqn:QRModelDecreaseLowerByStep}, and the second one from 
$\alphaK \leq \alphaLow$ and the definition of $\alphaLow$. 
The above equation implies that $\rho_k \geq \theta$ and therefore iteration $k$ is successful\footnote{For $\rho_k$ to be well-defined, we need the denominator to be strictly positive, which follows from \eqref{eqn:QRModelDecreaseLowerBySketchedGradient}.}. 
\end{proof}

The next Lemma shows that \autoref{alg:sketching_QR} satisfies \autoref{AA8}, and so also \autoref{AA4} due to \autoref{lem:AA8impliesAA4}.

\begin{lemma} \label{lem:qr:hbar_bound}
\autoref{alg:sketching_QR} satisfies \autoref{AA8} with 
\begin{equation}
    \barH{z_1}{z_2} = \frac{z_1^2}{2\alphaMax} \left[\sMax \bracket{\BMax + z_2^{-1}} + \kappaT\right]^{-2}, \label{eqn:QRHBarEq}
\end{equation}
where $\sMax$ is defined in \eqref{eqn::sMax}  and where the choice of $B_k$ is such that $\normTwo{B_k} \leq \BMax$ for all $k$, and some constant $\BMax\geq 0$. 
\end{lemma}

\begin{proof}
Let iteration $k$ be true. Using the definition of $\hat{q}_k$, we have 
   $ \grad \hat{q}_k(\hat{s}_k) = S_k \gradFK + S_k B_k S_k^T \sKHat + \oneOverAlphaK S_k S_k^T \sKHat$.
It follows that
\begin{align}
    \normTwo{S_k \gradFK} &= \normTwo{- S_k\bracket{ B_k + \oneOverAlphaK} S_k^T \sKHat + \grad \hat{q}_k(\hat{s}_k)} \notag\\
    & \leq \bracket{\sMax \bracket{\BMax + \oneOverAlphaK}}\normTwo{S_k^T \sKHat} + \normTwo{\grad \hat{q}_k(\hat{s}_k)} \notag\\
    & \leq \bracket{\sMax \bracket{\BMax + \oneOverAlphaK} + \kappaT} \normTwo{S_k^T \sKHat}, \label{eqn:QRSketchedGradientUpperBySketchedStep}
\end{align}
where we used $\normTwo{S_k}\leq \sMax$ on true iterations and $\normTwo{B_k} \leq \BMax$ to derive the first inequality and \eqref{eqn:QRskHatCond1} to derive the last inequality.
Therefore, using \eqref{eqn:QRModelDecreaseLowerByStep} and \eqref{eqn:QRSketchedGradientUpperBySketchedStep}, we have 
\begin{align}
     \mKHat{0} - \mKHat{\sKHat} &\geq \oneOverTwoAlphaK \normTwo{\SKTransposed\sKHat}^2 
     \geq \oneOverTwoAlphaK \bracket{\sMax \bracket{\BMax + \oneOverAlphaK} + \kappaT}^{-2} \normTwo{S_k \gradFK}^2 \notag\\
     &\geq \frac{1}{2\alphaMax} \bracket{\sMax \bracket{\BMax + \oneOverAlphaK} + \kappaT}^{-2} \normTwo{S_k \gradFK}^2,
     \label{eqn:QRModelDecreaseLowerBySketchedGradient}
\end{align}
satisfying \autoref{AA8}.
\end{proof}

\subsection{Iteration complexity of random-subspace quadratic regularisation methods}\label{iter_complexity_QR}
By specifying the choice of the random ensemble $\mathcal{S}$ that generates the random subspace in \autoref{alg:sketching_QR}, we can detail the complexity bounds even further, to capture their explicit dependence on the expressions we gave for $S_{max}$ and subspace dimension $l$ in the previous section. In addition to the three choices of matrices below, many other possibilities are allowed--such as sparse $s$-hashing matrices, orthogonal ensembles--but we do not detail them here for brevity. 

Applying \autoref{lem:AA8impliesAA4}, \autoref{lem:QRAlpahLow}, \autoref{lem:qr:hbar_bound} for \autoref{alg:sketching_QR}, we have that \autoref{AA3} and \autoref{AA4} are satisfied with
\begin{align}
    & \alphaLow = \frac{1-\theta}{L + \BMax} \notag \\
    & h(\epsilon, \alphaZero\gammaOne^{c+\newL}) = \theta \barH{(1 - \epS)^{1/2}\epsilon}{\alphaZero\gammaOne^{c+\newL}} 
     = \frac{\theta   (1-\epS)\epsilon^2}{2\alphaMax \bracket{\sMax \bracket{\BMax + \alphaZero^{-1}\gammaOne^{-c-\newL}} + \kappaT}^2} 
    \label{tmp-2022-1-13-2} 
\end{align}
Moreover, \autoref{AA5} is satisfied for \autoref{alg:sketching_QR}  by \autoref{lem:AA8impliesAA5}.
The following three subsections give complexity bounds for \autoref{alg:sketching_QR} with different random ensembles (whose properties are summarized in \autoref{tab:alg:sketching}). 

\subsubsection{Using scaled Gaussian matrices}
\autoref{alg:sketching_QR} with scaled Gaussian matrices of size $l=\mathcal{O}(1)$ as the choice for $S_k$ has, with high-probability, an iteration complexity $\mathO{\frac{d}{l}\epsilon^{-2}}$ to drive $\|\gradFK\|$ below $\epsilon$; the choice of the subspace dimension $l$ can be  a (small) (problem dimension-independent) constant (see \autoref{tab:alg:sketching}). 
\begin{theorem}\label{thm:complexity-QR-Gaussian}
Assume that the objective function $f$ in problem \eqref{eqn::fStar} is continuously differentiable with $L$-Lipschitz continuous gradient. 
Let $\deltaSTwo, \epS, \deltaOne>0$, $l\in \N^+$ be such that
\begin{equation*}
    \deltaS < \frac{c}{(c+1)^2}, \quad \nPreFactorTR >0,
\end{equation*}
where $\deltaS = e^{-l\epS^2/4} + \deltaSTwo$.
Apply \autoref{alg:sketching_QR} to \eqref{eqn::fStar}, where $B_k$ is such that $\normTwo{B_k} \leq \BMax$ for all $k$ and some constant $\BMax\geq 0$, and where $\cal{S}$ is the distribution of scaled Gaussian matrices (\autoref{def:Gaussian}).  Assume \autoref{alg:sketching_QR} runs for $N$ iterations such that
\begin{equation}
    N \geq  \nPreFactorTR \squareBracket{
         \fZeroMinusfStarOverH
         + \frac{\newL}{1+c}}, \notag
\end{equation}
where 
\begin{equation}
    h(\epsilon, \alphaZero\gammaOne^{c+\newL}) = \frac{\theta   (1-\epS)\epsilon^2}{ 2\alphaMax} \bracket{\squareBracket{1 + \sqrt{\frac{d}{l}} + \sqrt{2l^{-1}\logOneOverDeltaSTwo}}\bracket{\BMax + \alphaZero^{-1}\gammaOne^{-c-\newL}} + \kappaT}^{-2}  \notag
\end{equation}
and $\newL$ is given in \eqref{eqn:tauLDef}.
Then
\begin{equation}
    \probability{N \geq \nEps} \geq 1 - \chernoffLowerExponential, \notag
\end{equation}
where $\nEps$ is defined in \eqref{eqn::nEps}.
\end{theorem}

\begin{proof}
We note that \autoref{alg:sketching_QR} is a particular variant of \autoref{alg:generic}, and so \autoref{thm2} applies. Moreover, \autoref{AA3}, \autoref{AA4} and \autoref{AA5} are satisfied conform our earlier discussions. Applying \autoref{lem::deduceAA2}, \autoref{lem:GaussJLEmbedding} and \autoref{Lem:GaussSMax} for scaled Gaussian matrices, \autoref{AA2} is satisfied with 
\begin{align*}
    & \sMax = 1 + \sqrt{\frac{d}{l}} + \sqrt{2l^{-1}\logOneOverDeltaSTwo} \\
    & \deltaS =  e^{-\epS^2 l/ 4} + \deltaSTwo.
\end{align*}
Applying \autoref{thm2} and substituting the above expression of $\sMax$  in \eqref{tmp-2022-1-13-2} gives the desired result.
\end{proof}

\subsubsection{Using stable $1$-hashing matrices}
\autoref{alg:sketching_QR} with stable $1$-hashing matrices of size $l=\mathcal{O}(1)$ has, with high-probability, an iteration complexity $\mathO{\frac{d}{l}\epsilon^{-2}}$ to drive $\|\gradFK\|$ below $\epsilon$; the choice of the subspace dimension $l$ can be a (small) (problem dimension-independent) constant (see \autoref{tab:alg:sketching}). 

\begin{theorem}\label{thm:complexity-QR-stable-1-hashing}
Assume that the objective function $f$ in problem \eqref{eqn::fStar} is continuously differentiable with $L$-Lipschitz continuous gradient. 
Let $\deltaOne>0$, $\epS \in (0,3/4)$, $l\in \N^+$  be such that
\begin{equation*}
    \deltaS < \frac{c}{(c+1)^2}, \quad \nPreFactorTR >0,
\end{equation*}
where $\deltaS = e^{-\frac{l(\epS-1/4)^2}{C_3}}$ and $C_3$ is defined in \autoref{lem:tmp:2022-1-13-1}.
Apply \autoref{alg:sketching_QR} to \eqref{eqn::fStar}, where $B_k$ is such that $\normTwo{B_k} \leq \BMax$ for all $k$ and some constant $\BMax\geq 0$, and where $\cal{S}$ is the distribution of scaled stable 1-hashing matrices (\autoref{def::stable-1-hashing}).  Assume \autoref{alg:sketching_QR} runs for $N$ iterations such that
\begin{equation}
 N \geq  \nPreFactorTR \squareBracket{
         \fZeroMinusfStarOverH
         + \frac{\newL}{1+c}}, \notag
\end{equation}
where 
\begin{equation}
    h(\epsilon, \alphaZero\gammaOne^{c+\newL}) = \frac{\theta   (1-\epS)\epsilon^2}{ 2\alphaMax} \bracket{\sqrt{\ceil{d/l}}\bracket{\BMax + \alphaZero^{-1}\gammaOne^{-c-\newL}} + \kappaT}^{-2}  \notag
\end{equation}
and $\newL$ is given in \eqref{eqn:tauLDef}.
Then, we have 
\begin{equation}
    \probability{N \geq \nEps} \geq 1 - \chernoffLowerExponential, \notag
\end{equation}
where $\nEps$ is defined in \eqref{eqn::nEps}.
\end{theorem}

\begin{proof}
Applying \autoref{lem::deduceAA2}, \autoref{lem:tmp:2022-1-13-1} and \autoref{lem:stable-1-hashing-SMax} for stable 1-hashing matrices, \autoref{AA2} is satisfied with 
    $\sMax = \sqrt{\ceil{d/l}}$ and
    $\deltaS =  e^{-\frac{l(\epS-1/4)^2}{C_3}}$.
Applying \autoref{thm2} and substituting the expression of $\sMax$ above in \eqref{tmp-2022-1-13-2} gives the desired result. 
\end{proof}

\subsubsection{Using sampling matrices}
\autoref{alg:sketching_QR} with scaled sampling matrices of size $l$ has, with high-probability, an iteration complexity 
$\mathO{\frac{d}{l}\epsilon^{-2}}$ to drive $\|\gradFK\|$ below $\epsilon$. However, unlike in the above two cases, here $l$ depends on the problem dimension $d$ and a problem-dependent constant $\nu$ that reflects how similar in magnitude the entries in $\nabla f(x)$ are (see \autoref{tab:alg:sketching}). If $\nu = \mathO{1/d}$ (and so these entries are similar in size), then $l=\mathcal{O}(1)$ in a problem dimension-independent way; else, indeed, we need to choose $l$ proportional to $d$.

\begin{theorem}\label{thm:complexity-QR-sampling}
Assume that the objective function $f$ in problem \eqref{eqn::fStar} is continuously differentiable with $L$-Lipschitz continuous gradient. 
Let $\deltaOne>0$, $\epS \in (0,1)$, $l\in \N^+$ be such that
\begin{equation*}
    \deltaS < \frac{c}{(c+1)^2}, \quad \nPreFactorTR >0,
\end{equation*}
where $\deltaS =e^{- \frac{\epS^2 l}{2d\nu^2}}$ and $\nu$ is defined in \eqref{eq:nu_def}.
Apply \autoref{alg:sketching_QR} to \eqref{eqn::fStar}, where $B_k$ is such that $\normTwo{B_k} \leq \BMax$ for all $k$ and some constant $\BMax\geq 0$, and where $\cal{S}$ is the distribution of scaled sampling matrices (\autoref{def:sampling}).  Assume \autoref{alg:sketching_QR} runs for $N$ iterations such that
\begin{equation}
    N \geq  \nPreFactorTR \squareBracket{
         \fZeroMinusfStarOverH
         + \frac{\newL}{1+c}}, \notag
\end{equation}
where 
\begin{equation}
    h(\epsilon, \alphaZero\gammaOne^{c+\newL}) = \frac{\theta   (1-\epS)\epsilon^2}{ 2\alphaMax} \bracket{\sqrt{d/l}\bracket{\BMax + \alphaZero^{-1}\gammaOne^{-c-\newL}} + \kappaT}^{-2}  \notag
\end{equation}
and $\newL$ is given in \eqref{eqn:tauLDef}.
Then, we have 
\begin{equation}
    \probability{N \geq \nEps} \geq 1 - \chernoffLowerExponential, \notag
\end{equation}
where $\nEps$ is defined in \eqref{eqn::nEps}.
\end{theorem}

\begin{proof}
Applying \autoref{lem::deduceAA2}, \autoref{lem:sampling:non-uniformity-BCGN} and \autoref{tmp-2022-1-14-12} for scaled sampling matrices, \autoref{AA2} is satisfied with 
$\sMax = \sqrt{d/l}$ and 
    $\deltaS = e^{- \frac{\epsilon^2 l}{2d\nu^2}}$.
Applying \autoref{thm2} and substituting the expression of $\sMax$ above in \eqref{tmp-2022-1-13-2} gives the desired result.
\end{proof}

\begin{remark}
The dependency on $\epsilon$ in the iteration complexity matches that of the full-dimensional quadratic regularisation method. 
Note that for each ensemble considered, there is dimension-dependence in the iteration bound of the form $\frac{d}{l}$; but not in the size of the sketching projection in the case of Gaussian and certain sparse ensembles.
We may eliminate the dependence on $d$ in the iteration complexity bound by fixing the ratio $\frac{d}{l}$ to be a constant, so that $d$ and $l$
grow proportionally.
\end{remark}

\subsection{A random-subspace trust region method}
Here we present a  random-subspace trust region method with sketching, \autoref{alg:sketching_TR}, which is a particular form of \autoref{alg:sketching} where the step is computed using a trust region approach. The general structure of this section and the main results mirror those in the previous subsection on quadratic regularization; we include them here in order to illustrate our framework using another state of the art strategy and so that the precise details and constants can be given in full.

\begin{algorithm}[H]
\begin{description}
 \item[Initialization] \ \\
 Choose a matrix distribution $\cal{S}$ of matrices $S \in \R^{l\times d}$. 
 Choose constants $\gamma_1\in (0,1)$, $\gamma_2 = \gammaOne^{-c}$, for some $c \in \N^+$, $l \in \N^+$, $\theta \in (0,1)$ and $\alpha_{\max}>0$.
 Initialize the algorithm by setting $x_0 \in \R^d$, $\alpha_0 = \alphaMax \gamma_1^p$ for some $p \in \N^+$ and $k=0$.
 \item[1. Compute a reduced model and a step]\  \\
 Draw a random matrix $S_k \in \R^{l \times d}$ from $\cal{S}$, and let
 \begin{align*}
     \mkHatSHat = \fK + \innerProduct{\sKGradFK}{\sHat} + \frac{1}{2} \innerProduct{\sHat}{\sKBKSKT \sHat}
 \end{align*}
 where $B_k \in \R^{d\times d}$ is a positive semi-definite user-chosen matrix (the choice $B_k=0$ is allowed).\\
 Compute $\sKHat$ by approximately minimising $\mkHatSHat$ such that for some $C_7>0$,
\begin{align}
    &\normTwo{\sKHat} \leq \alphaK\quad {\rm and} \label{tmp-2021-12-31-4} \\
    &\mKHat{0} - \mKHat{\sKHat} \geq C_7 \normTwo{S_k \gradFK} \min \set{\alphaK, \frac{\normTwo{S_k \gradFK}}{\normTwo{B_k}}}, \label{eqn::TRStepComputationCOndition}
\end{align}
and set $\sK = S_k^T \sKHat$.
\item[2. Check sufficient decrease]\ \\  
Check the sufficient decrease condition 
\begin{equation}
    \fK - \fKPlusOne \geq \theta \squareBracket{\mKHat{0} - \mKHat{\hat{\sK}}}. \notag
\end{equation}
\item[3. Update the parameter $\alphaK$ and possibly take the trial step $\sK$]\ \\
If sufficient decrease is achieved, set $\xKPlusOne = \xK + \sK$ and $\alphaKPlusOne = \min \set{\alphaMax, \gammaTwo\alphaK}$ [successful]. \\
Otherwise set $\xKPlusOne = \xK$ and $\alphaKPlusOne = \gammaOne \alphaK$ [unsuccessful].\\
Increase the iteration count by setting $k=k+1$ in both cases.\\
\caption{\bf{A random-subspace trust region algorithm using sketching}} \label{alg:sketching_TR} 
\end{description}
\end{algorithm}

 \autoref{alg:sketching_TR} is   \autoref{alg:sketching} with full details of the calculation of the reduced step $\hat{s}_k$, its ambient-space projection $s_k$ and definition of the parameter $\alpha_k$.

\begin{remark}
Lemma 4.3 in \cite{Nocedal:2006uv} shows that there  exists $\hat{s_k} \in \R^l$ such that \eqref{eqn::TRStepComputationCOndition} holds. 
In particular, letting the model gradient $\gK = \sKGradFK$,
if $\gK=0$, we set $\hat{s_k} = 0$; otherwise we may let $\sKHat$ be the Cauchy point (that is, the point where the model $\hat{m}_k$ is minimised in the negative model gradient direction within the trust region), which can be easily computed.
\end{remark}

\autoref{tmp-2021-12-31-5} shows that \autoref{alg:sketching_TR} satisfies \autoref{AA3}.

\begin{lemma} \label{tmp-2021-12-31-5}
Assume that the objective function $f$ in problem \eqref{eqn::fStar} is continuously differentiable with $L$-Lipschitz continuous gradient. Let
  \autoref{alg:sketching_TR} be applied to \eqref{eqn::fStar}, where the choice of $B_k$ is such that $\normTwo{B_k} \leq \BMax$ for all $k$, and some constant $\BMax\geq 0$.
Then \autoref{alg:sketching_TR} satisfies \autoref{AA3} with
\begin{equation}
    \alphaLow = \oneMiusEpsSToHalf \epsilon \min \bracket{ \frac{C_7 (1-\theta)}{(\lPlusHalfBmax)\sMax^2}, \frac{1}{\BMax} }.\label{eqn::TRalphaLowExpression}
\end{equation}
\end{lemma}

\begin{proof}
Let $\epsilon>0$ and $k < \nEps$, and assume iteration $k$ is true with $\alphaK \leq \alphaLow$, define
    $\rho_k = [f(x_k) - f(x_k + s_k)]/[\mKHat{0} - \mKHat{\sKHat}]$.
Then we have
\begin{align}
    \abs{1-\rho_k} 
    & \leq \frac{(\lPlusHalfBmax)\normTwo{\SKTransposedsKHat}^2}{C_7 \normTwo{\sKGradFK}\min \bracket{\alphaK, \frac{\normTwo{\sKGradFK}}{\normTwo{B_k}}}}
    \leq \frac{(\lPlusHalfBmax) \sMax^2 \alphaK^2}{C_7 \normTwo{\sKGradFK}\min \bracket{\alphaK, \frac{\normTwo{\sKGradFK}}{\normTwo{B_k}}}} 
    \notag\\
    & \leq \frac{(\lPlusHalfBmax)\sMax^2\alphaK^2}{C_7 \oneMiusEpsSToHalf\epsilon \min \bracket{\alphaK, \frac{\oneMiusEpsSToHalf\epsilon}{\BMax}}}  \leq 1 - \theta,\notag
\end{align}
where the first inequality follows from \eqref{eqn::TRStepComputationCOndition} and \autoref{lem:Taylor}, the second inequality follows from \eqref{eqn::sMax} and $\normTwo{\sKHat} \leq \alphaK$, the third inequality follows from \eqref{eqn::JL} and the fact that $\gradFK>\epsilon$ for $k<\nEps$, while the last inequality follows from $\alphaK \leq \alphaLow$ and \eqref{eqn::TRalphaLowExpression}. It follows then that $\rho_k \geq \theta$ and so the iteration $k$ is successful\footnote{Note that for $k$ being a true iteration with $k<\nEps$, \eqref{eqn::TRStepComputationCOndition} along with \eqref{eqn::JL}, $\alphaK>0$ gives $\mKHat{0}-\mKHat{\sKHat}>0$ so that $\rho_k$ is well defined.}. 
\end{proof}

The next lemma shows that \autoref{alg:sketching_TR} satisfies \autoref{AA8}, thus satisfying \autoref{AA4}.

\begin{lemma}\label{tmp-2022-1-13-5}
\autoref{alg:sketching_TR} satisfies \autoref{AA8} with 
\begin{equation}
    \barH{z_1}{z_2} = C_7\min \bracket{z_1 z_2, z_1^2/\BMax},\notag
\end{equation}
where the choice of $B_k$ is such that $\normTwo{B_k} \leq \BMax$ for all $k$, and some constant $\BMax\geq 0$. 
\end{lemma}
\begin{proof}
Use \eqref{eqn::TRStepComputationCOndition} with $\normTwo{B_k} \leq \BMax$. 
\end{proof}

\subsection{Iteration complexity of random-subspace trust region methods}

Here we derive complexity results for three concrete implementations of \autoref{alg:sketching_TR} that use different random ensembles. The exposition follows closely that in Section \ref{iter_complexity_QR},  and the complexity results are of the same order in $\epsilon$ and $\frac{d}{l}$ as for the quadratic regularization algorithms, namely, $\mathcal{O}\left(\frac{d}{l}\epsilon^{-2}\right)$, but with different constants. 

Applying \autoref{lem:AA8impliesAA4}, \autoref{tmp-2021-12-31-5}, \autoref{tmp-2022-1-13-5} for \autoref{alg:sketching_TR}, we have that \autoref{AA3} and \autoref{AA4} are satisfied with
\begin{align}
    & \alphaLow = \oneMiusEpsSToHalf \epsilon \min \bracket{ \frac{C_7 (1-\theta)}{(\lPlusHalfBmax)\sMax^2}, \frac{1}{\BMax} } \notag \\
    & h(\epsilon, \alphaZero\gammaOne^{c+\newL}) = \theta \barH{(1 - \epS)^{1/2}\epsilon}{\alphaZero\gammaOne^{c+\newL}} 
     = \theta C_7  \minMe{\bracket{1-\epS}^{1/2} \epsilon \alphaZero \gammaOne^{c + \newL}}{\bracket{1-\epS} \epsilon^2/\BMax }
    \label{tmp-2022-1-13-6} 
\end{align}
Here, unlike in the analysis of \autoref{alg:sketching_QR}, $\alphaLow$ (and consequently $\newL$) depend on $\epsilon$, which we now make explicit. Using the definition of $\newL$ in \eqref{eqn:tauLDef} and substituting in the expression for $\alphaLow$, we have
\begin{align*}
    \alphaZero \gammaOne^{c+\newL} 
    & = \alphaZero\gammaOne^c
        \gammaOne^{
            \ceil{
                \log_\gammaOne \bracket{
                    \minMe{
                        \oneMiusEpsSToHalf \epsilon \min \bracket{ \frac{C_7 (1-\theta)}{(\lPlusHalfBmax)\sMax^2}, \frac{1}{\BMax} } \alphaZero^{-1}
                    }
                    {
                        \gammaTwo^{-1}
                    }
                }
            }
        } \notag \\
    & \geq \alphaZero\gammaOne^c \gammaOne
        \minMe{
                \oneMiusEpsSToHalf \epsilon \min \bracket{ \frac{C_7 (1-\theta)}{(\lPlusHalfBmax)\sMax^2}, \frac{1}{\BMax} } \alphaZero^{-1}
                }
                {
                    \gammaTwo^{-1}
                } \notag \\
    & = \gammaOne^{c+1}
        \minMe{
                \oneMiusEpsSToHalf \epsilon \min \bracket{ \frac{C_7 (1-\theta)}{(\lPlusHalfBmax)\sMax^2}, \frac{1}{\BMax} }
                }
                {
                    \alphaZero \gammaTwo^{-1}
                }, \notag 
\end{align*}
where we used $\ceil{y} \leq y+1$ to derive the inequality.
Therefore, \eqref{tmp-2022-1-13-6} implies 
\begin{align}
    &h(\epsilon, \alphaZero\gammaOne^{c+\newL}) \notag\\
    &\geq 
    \theta C_7 
    \minMe{
        \gammaOne^{c+1}
        \minMe{
                \bracket{1-\epS} \epsilon^2 \min \bracket{ \frac{C_7 (1-\theta)}{(\lPlusHalfBmax)\sMax^2}, \frac{1}{\BMax} }
                }
                {
                    \bracket{1-\epS}^{1/2} \epsilon \alphaZero \gammaTwo^{-1}
                }        
    }
    {
        \frac{\bracket{1-\epS}\epsilon^2}{\BMax}
    } \notag\\ 
    & = 
    \theta C_7 \bracket{1-\epS}\epsilon^2
    \minMe{
        \gammaOne^{c+1}
        \minMe{
                \min \bracket{ \frac{C_7 (1-\theta)}{(\lPlusHalfBmax)\sMax^2}, \frac{1}{\BMax} }
                }
                {
                    \frac{\alphaZero}{\bracket{1-\epS}^{1/2}\epsilon \gammaTwo}
                }        
    }
    {
        \frac{1}{\BMax}
    } \notag\\
    & = 
    \theta C_7 \bracket{1-\epS}\epsilon^2
        \gammaOne^{c+1}
        \minMe{
                \min \bracket{ \frac{C_7 (1-\theta)}{(\lPlusHalfBmax)\sMax^2}, \frac{1}{\BMax} }
                }
                {
                    \frac{\alphaZero}{\bracket{1-\epS}^{1/2}\epsilon \gammaTwo}
                }        
    \label{tmp-2022-1-13-9}.
\end{align}
where the last equality follows from $\gammaOne^{c+1} < 1$.
Moreover, \autoref{AA5} for \autoref{alg:sketching_TR} is satisfied by applying \autoref{lem:AA8impliesAA5}.
The following three subsections give complexity results of \autoref{alg:sketching_TR} using different random ensembles within \autoref{alg:sketching_TR}. Again, we suggest that the reader to refer back to \autoref{tab:alg:sketching} for a summary of the theoretical properties of these random ensembles. 

\subsubsection{Using scaled Gaussian matrices}
\autoref{alg:sketching_TR} with scaled Gaussian matrices have a (high-probability) iteration complexity of $\mathO{\frac{d}{l}\epsilon^{-2}}$ to drive $\gradFK$ below $\epsilon$, where $l$ can be chosen as a (problem dimension-independent) constant (see \autoref{tab:alg:sketching}). 
\begin{theorem}\label{thm:complexity-TR-Gaussian}
Assume that the objective function $f$ in problem \eqref{eqn::fStar} is continuously differentiable with $L$-Lipschitz continuous gradient. 
Let $\deltaSTwo, \epS, \deltaOne>0$, $l\in \N^+$ be such that
\begin{equation*}
    \deltaS < \frac{c}{(c+1)^2}, \quad \nPreFactorTR >0,
\end{equation*}
where $\deltaS = e^{-l\epS^2/4} + \deltaSTwo$.
Apply \autoref{alg:sketching_TR} to \eqref{eqn::fStar}, where $B_k$ is such that $\normTwo{B_k} \leq \BMax$ for all $k$ and some constant $\BMax\geq 0$, and where $\cal{S}$ is the distribution of scaled Gaussian matrices (\autoref{def:Gaussian}).  Assume \autoref{alg:sketching_TR} runs for $N$ iterations such that
\begin{equation}
    N \geq  \nPreFactorTR \squareBracket{
         \fZeroMinusfStarOverH
         + \frac{\newL}{1+c}}, \notag
\end{equation}
where 
\begin{equation}
    h(\epsilon, \alphaZero\gammaOne^{c+\newL}) 
    = \theta C_7 \bracket{1-\epS}\epsilon^2
        \gammaOne^{c+1}
        \min\left\{
                \frac{C_7 (1-\theta)}{[\lPlusHalfBmax]\sMax^2},
                   \frac{1}{\BMax},                
                    \frac{\alphaZero}{\bracket{1-\epS}^{1/2}\epsilon \gammaTwo} \right\}
\end{equation}
where $\sMax = 1 + \sqrt{\frac{d}{l}} + \sqrt{\frac{2}{l}\logOneOverDeltaSTwo}$. Then
\begin{equation}
    \probability{N \geq \nEps} \geq 1 - \chernoffLowerExponential, \notag
\end{equation}
where $\nEps$ is defined in \eqref{eqn::nEps}.
\end{theorem}

\begin{proof}
We note that \autoref{alg:sketching_TR} is a particular version of \autoref{alg:generic} therefore \autoref{thm2} applies. Applying \autoref{lem::deduceAA2}, \autoref{lem:GaussJLEmbedding} and \autoref{Lem:GaussSMax} for scaled Gaussian matrices, \autoref{AA2} is satisfied with $\sMax$ as above and
$\deltaS =  e^{-\epS^2 l/ 4} + \deltaSTwo$.
Applying \autoref{thm2} and substituting the expression of $\sMax$ in \eqref{tmp-2022-1-13-9} gives the desired result.
\end{proof}

\subsubsection{Using stable $1$-hashing matrices}
\autoref{alg:sketching_TR} with stable $1$-hashing matrices of size $l=\mathcal{O}(1)$ has, with high-probability, an iteration complexity $\mathO{\frac{d}{l}\epsilon^{-2}}$ to drive $\|\gradFK\|$ below $\epsilon$; the choice of the subspace dimension $l$ can be a (small) (problem dimension-independent) constant (see \autoref{tab:alg:sketching}). 

\begin{theorem}\label{thm:complexity-TR-stable-1-hashing}
Assume that the objective function $f$ in problem \eqref{eqn::fStar} is continuously differentiable with $L$-Lipschitz continuous gradient. 
Let $\deltaOne>0$, $\epS \in (0,3/4)$, $l\in \N^+$  be such that
\begin{equation*}
    \deltaS < \frac{c}{(c+1)^2}, \quad \nPreFactorTR >0,
\end{equation*}
where $\deltaS = e^{-\frac{l(\epS-1/4)^2}{C_3}}$ and $C_3$ is defined in \autoref{lem:tmp:2022-1-13-1}.
Apply \autoref{alg:sketching_TR} to \eqref{eqn::fStar}, where $B_k$ is such that $\normTwo{B_k} \leq \BMax$ for all $k$ and some constant $\BMax\geq 0$, and where $\cal{S}$ is the distribution of scaled stable 1-hashing matrices (\autoref{def::stable-1-hashing}).  Assume \autoref{alg:sketching_TR} runs for $N$ iterations such that
\begin{equation}
 N \geq  \nPreFactorTR \squareBracket{
         \fZeroMinusfStarOverH
         + \frac{\newL}{1+c}}, \notag
\end{equation}
where 
\begin{equation}
    h(\epsilon, \alphaZero\gammaOne^{c+\newL}) = 
        \theta C_7 \bracket{1-\epS}\epsilon^2
        \gammaOne^{c+1}
        \min\left\{
                 \frac{C_7 (1-\theta)}{(\lPlusHalfBmax)\ceil{d/l}}, \frac{1}{\BMax},
                               \frac{\alphaZero}{\bracket{1-\epS}^{1/2}\epsilon \gammaTwo}\right\}
                \notag
\end{equation}
Then, we have 
\begin{equation}
    \probability{N \geq \nEps} \geq 1 - \chernoffLowerExponential, \notag
\end{equation}
where $\nEps$ is defined in \eqref{eqn::nEps}.
\end{theorem}

\begin{proof}
Applying \autoref{lem::deduceAA2}, \autoref{lem:tmp:2022-1-13-1} and \autoref{lem:stable-1-hashing-SMax} for stable 1-hashing matrices, \autoref{AA2} is satisfied with 
$ \sMax = \sqrt{\ceil{d/l}}$ and 
    $ \deltaS =  e^{-\frac{l(\epS-1/4)^2}{C_3}}$.
Applying \autoref{thm2} and substituting the expression of $\sMax$ in \eqref{tmp-2022-1-13-9} gives the desired result.
\end{proof}

\subsubsection{Using sampling matrices}
\autoref{alg:sketching_TR} with scaled sampling matrices of size $l$ has, with high-probability, an iteration complexity 
$\mathO{\frac{d}{l}\epsilon^{-2}}$ to drive $\|\gradFK\|$ below $\epsilon$. However, unlike in the above two cases, here $l$ depends on the problem dimension $d$ and a problem-dependent constant $\nu$ that reflects how similar in magnitude the entries in $\nabla f(x)$ are (see \autoref{tab:alg:sketching}). If $\nu = \mathO{1/d}$ (and so these entries are similar in size), then $l=\mathcal{O}(1)$ in a problem dimension-independent way; else, indeed, we need to choose $l$ proportional to $d$.

\begin{theorem}\label{thm:complexity-TR-sampling}
Assume that the objective function $f$ in problem \eqref{eqn::fStar} is continuously differentiable with $L$-Lipschitz continuous gradient. 
Let $\deltaOne>0$, $\epS \in (0,1)$, $l\in \N^+$ be such that
\begin{equation*}
    \deltaS < \frac{c}{(c+1)^2}, \quad \nPreFactorTR >0,
\end{equation*}
where $\deltaS =e^{- \frac{\epS^2 l}{2d\nu^2}}$ and $\nu$ is defined in \eqref{eq:nu_def}.
Apply \autoref{alg:sketching_TR} to \eqref{eqn::fStar}, where $B_k$ is such that $\normTwo{B_k} \leq \BMax$ for all $k$ and some constant $\BMax\geq 0$, and where $\cal{S}$ is the distribution of scaled sampling matrices (\autoref{def:sampling}).  Assume \autoref{alg:sketching_TR} runs for $N$ iterations such that
\begin{equation}
    N \geq  \nPreFactorTR \squareBracket{
         \fZeroMinusfStarOverH
         + \frac{\newL}{1+c}}, \notag
\end{equation}
where 
\begin{equation}
    h(\epsilon, \alphaZero\gammaOne^{c+\newL}) = 
        \theta C_7 \bracket{1-\epS}\epsilon^2
        \gammaOne^{c+1}
        \minMe{
                \min \bracket{ \frac{C_7 (1-\theta)}{(\lPlusHalfBmax)d/l}, \frac{1}{\BMax} }
                }
                {
                    \frac{\alphaZero}{\bracket{1-\epS}^{1/2}\epsilon \gammaTwo}
                }\notag
\end{equation}
Then, we have 
\begin{equation}
    \probability{N \geq \nEps} \geq 1 - \chernoffLowerExponential, \notag
\end{equation}
where $\nEps$ is defined in \eqref{eqn::nEps}.
\end{theorem}

\begin{proof}
Applying \autoref{lem::deduceAA2}, \autoref{lem:sampling:non-uniformity-BCGN} and \autoref{tmp-2022-1-14-12} for scaled sampling matrices, \autoref{AA2} is satisfied with 
$\sMax = \sqrt{d/l}$ and
    $ \deltaS = e^{- \frac{\epsilon^2 l}{2d\nu^2}}$.
Applying \autoref{thm2} and substituting the expression of $\sMax$ in \eqref{tmp-2022-1-13-9} gives the desired result.
\end{proof}

\begin{remark}
Similarly to \autoref{alg:sketching_QR}, \autoref{alg:sketching_TR} matches the iteration complexity $\mathcal{O}(\epsilon^{-2})$ of the corresponding (full-space) trust region method in terms of desired accuracy $\epsilon$.  Furthermore, \autoref{alg:sketching_QR} and \autoref{alg:sketching_TR} with a(ny) of the above three random ensembles only require $l$ directional derivative evaluations of $f$ per iteration, instead of $d$ such evaluations required by the (full-space) methods. The dimension $l$ of the projection subspace can also be chosen independent of $d$, in which case, a lower computational complexity and memory requirement can be gained per iteration. However, the iteration complexity of these subspace variants increases by a factor of $d/l$, which could be eliminated if $l$ is a multiple of $d$; with the case $l=d$ recovering the full-dimensional first-order/trust-region or quadratic regularization complexity bound, which is reassuring for our theory. 

\end{remark}

\section{Random-subspace Gauss-Newton methods for solving nonlinear least-squares problems: theoretical and numerical illustrations}
In this section, we further illustrate our subspace algorithms and results by particularizing our approach to  nonlinear least squares problems of the form,
\begin{equation}\label{NLS}
\min_{x \in \R^d} f(x) = \frac{1}{2}\sum_{i=1}^n \norm{ r_i(x) }_2^2=\frac{1}{2}\norm{r(x)}_2^2 
\end{equation}
where $r = (r_1, \dots, r_n): \R^d \to \R^n$ is a smooth vector of nonlinear (possibly nonconvex) residual functions, and
$J(x)$ is the $n\times d$ matrix of first derivatives of $r(x)$.
The classical Gauss-Newton (GN) algorithm \cite{Nocedal:2006uv} is an approximate second-order method that at 
every iterate $x_k$, approximately minimises the following convex quadratic local model
\begin{align*}
f(x_k) + \dotp{ J(x_k)^Tr(x_k)}{ s } + \frac{1}{2} \dotp{ s }{ J(x_k)^T J(x_k) s }  
\end{align*}
over $s \in \R^d$, which is the same as the linear least squares  $\frac{1}{2}\|r(x_k) + J(x_k)s\|^2$. 
In our approach, which we refer to as Random-Subspace Gauss-Newton (RS-GN), we reduce the dimensionality of this model by minimising it in an $l$-dimensional random subspace, with $l\ll d$, which gives the following reduced model,
\begin{align}
\hat{m}_k(\sHat)=f(x_k) + \dotp{ J_{\mathcal{S}}(x_k)^Tr(x_k) }{ \sHat } + \frac{1}{2} \dotp{ \sHat }{ J_{\mathcal{S}}(x_k)^T J_{\mathcal{S}}(x_k) \sHat }, \quad \sHat \in \R^l, 
\label{m_k_definition_TR}
\end{align}
where  $J_{\mathcal{S}}(x_k) = J(x_k)S_k^T \in \R^{n\times l}$ denotes the reduced Jacobian for $S_k \in \R^{l\times d}$ being a randomly generated sketching matrix. Letting $B_k = J_k^T J_k$, and recalling that 
$\grad f(x) = J(x)^Tr(x)$, we deduce that $\hat{m}_k(\sHat)$ in \eqref{m_k_definition_TR} coincides with the reduced model in \eqref{eqn::mKHatSpec}.
Thus the 
 \autoref{alg:sketching} framework can be applied directly to \eqref{NLS}, with this choice of $\hat{m}_k$ in \eqref{m_k_definition_TR}.
In particular,  quadratic regularization variants (\autoref{alg:sketching_QR}) and trust region ones (\autoref{alg:sketching_TR}) can be straightforwardly devised; with the ensuing convergence and complexity guarantees of Theorems \ref{thm:complexity-QR-Gaussian}--\ref{thm:complexity-QR-sampling} for \autoref{alg:sketching_QR}, and those of Theorems \ref{thm:complexity-TR-Gaussian}--\ref{thm:complexity-TR-sampling} for \autoref{alg:sketching_TR}, holding under usual sketching assumptions and whenever the Jacobian $J(x)$ is uniformly bounded above and $r(x)$ is Lipschitz continuous  (which ensures $B_k$ is uniformly bounded above and $\nabla f$ is Lipschitz continuous).  

Compared to the classical Gauss-Newton model, in addition to the speed-up gained due to the model dimension being reduced from $d$ to $l$, the reduced model \eqref{m_k_definition_TR} also offers the computational advantage that it only needs to evaluate $l$ Jacobian actions, giving $J_{\mathcal{S}}(x_k)$, instead of the full Jacobian matrix $J(x_k)$. 
In its simplest form, when $S_k$ is a scaled sampling matrix, $J_{\mathcal{S}}$ can be thought of as a random subselection of columns of the full Jacobian $J$, which leads to variants of our framework that are Block-Coordinate Gauss-Newton (BC-GN) methods. In this case, for example, if the Jacobian were being calculated by finite-differences of the residual $r$, only a small number of evaluations of $r$ along coordinate directions would be needed;  such a BC-GN variant has already been used for parameter estimation in climate modelling \cite{tett2017calibrating}. Note that theoretically, the convergence of BC-GN method requires an upper bound on $\frac{\normInf{\gradFK}}{\normTwo{\gradFK}}$ for all $k\in \N$ (for more details, see the discussion on sampling matrices on page \pageref{sampling_mat_paragraph}) and \autoref{thm:complexity-TR-sampling}.
More generally, $S_k$ can be generated from any matrix distribution that satisfies \autoref{AA6}, \autoref{AA7}, such as scaled Gaussian matrices or $s$-hashing matrices. We will now proceed to testing BC-GN and RS-GN on standard nonlinear least squares test problems.

\subsection{Numerical experiments}
In this section, we numerically test RS-GN with trust-region using different choices of the sketching matrix $S_k$ (results for quadratic regularisation are comparable); the code can be found at \texttt{https://github.com/jfowkes/BCGN}.  We use suitable subsets of the extensive CUTEst test collection \cite{gould2015cutest}. We measure performance using data profiles, a variant of performance profiles \cite{dolan2002benchmarking}, over Jacobian actions (namely, Jacobian matrix-vector multiply\footnote{We note that if the given vector is a coordinate one, then the ensuing Jacobian action is a column of the Jacobian, which is the derivative of each component of $r(x)$ with respect to one coordinate in $x$.}) and runtime. That is, for each solver $s$, each test problem $p \in \mathcal{P}$ and for an accuracy level $\tau\in(0,1)$, we determine the number of Jacobian action evaluations $N_p(s,\tau)$ required for a problem to be solved: 
\[
N_p(s,\tau) \deq \text{no.\ Jacobian action evaluations required until } f(x_k) \le f^* + \tau(f(x_0) - f^*), 
\]
where $f^*$ is an estimate of the true minimum $f(x^*)$. (Note that sometimes $f^*$ is taken to be the best value achieved by any solver.) We define $N_p(s,\tau)=\infty$ if this was not achieved in the maximum computational budget allowed, which we take to be $50d_p$ Jacobian action evaluations where $d_p$ is the dimension of test problem $p$.
To obtain data profiles, we can then normalise $N_p(s,\tau)$ by the problem dimension $d_p$. That is, we plot
\[
\pi_{s,\tau}^N(\alpha) \deq \frac{\abs{\{p\in\mathcal{P} : N_p(s,\tau) \le \alpha 
d_p\}}}{\abs{\mathcal{P}}}, \qquad \text{for } \alpha \in [0,50].
\]
We perform $100$ runs of each RS-GN variant (since RS-GN is a randomised algorithm) and thus to enable fair comparison, we treat each run as a separate `problem' in the above.

\paragraph{Zero Residual CUTEst Problems}
First we consider a test set of 32 zero-residual ($f(x^*)=0$) nonlinear least squares problems from CUTEst, given in \autoref{tab:cutest32}. These problems mostly have dimension around $d=100$, although one or two have $d=50$ to $64$. 
We view the remit of subspace methods as enabling progress of the algorithm from little/less problem information than full-dimensional variants. Thus of particular interest is to test the methods' behaviour in the low accuracy regime (namely, $\tau=10^{-1}$ in the data profiles). 

\begin{table}[!h]
\centering
\begin{tabular}{lrrlrrlrr}  
\toprule
Name & $d$ & $n$ & Name & $d$ & $n$ & Name & $d$ & $n$ \\
\midrule
ARGTRIG & 100 & 100 & DRCAVTY1 & 100 & 100 & OSCIGRNE & 100 & 100 \\
ARTIF & 100 & 100 & DRCAVTY3 & 100 & 100 & POWELLSE & 100 & 100 \\
BDVALUES & 100 & 100 & EIGENA & 110 & 110 & SEMICN2U & 100 & 100 \\
BRATU2D & 64 & 64 & EIGENB & 110 & 110 & SEMICON2 & 100 & 100 \\
BROWNALE & 100 & 100 & FLOSP2TL & 59 & 59 & SPMSQRT & 100 & 164 \\
BROYDN3D & 100 & 100 & FLOSP2TM & 59 & 59 & VARDIMNE & 100 & 102 \\
BROYDNBD & 100 & 100 & HYDCAR20 & 99 & 99 & LUKSAN11 & 100 & 198 \\
CBRATU2D & 50 & 50 & INTEGREQ & 100 & 100 & LUKSAN21 & 100 & 100 \\
CHANDHEQ & 100 & 100 & MOREBVNE & 100 & 100 & YATP1NE & 120 & 120 \\
CHEMRCTA & 100 & 100 & MSQRTA & 100 & 100 & YATP2SQ & 120 & 120 \\
CHNRSBNE & 50 & 98 & MSQRTB & 100 & 100 & & & \\
\bottomrule
\end{tabular}
\caption{The 32 CUTEst test problems in the zero-residual test set.}\label{tab:cutest32}
\end{table}

Figures \ref{fig:cutest32_coordinate}, \ref{fig:cutest32_gaussian} and \ref{fig:cutest32_hashing} show the performance of RS-GN with different choices of sketching matrices (sampling, Gaussian and $3$-hashing, respectively) and for different sizes of the subspace $l$ 
(as a fraction of $d$), in data profiles that measure cumulative Jacobian actions evaluations (meaning total number of Jacobian matrix-vector products that are used). We find that methods perform as expected: the more problem information, the more accuracy can be achieved and so the full-dimensional Gauss-Newton performs best, with the subspace variants decreasing in performance as the amount of per iteration problem information decreases. It is also useful to compare across the different sketching matrices and to this end we plot the Jacobian actions data profile for a subspace size of $0.75d$ in \autoref{fig:cutest32_across}; we see that interestingly, sampling matrices exhibit the best performance from a Jacobian action budget perspective, in this low-accuracy regime.


\begin{figure}[!p]
\centering
\begin{minipage}{.45\textwidth}
\centering
\includegraphics[width=\textwidth]{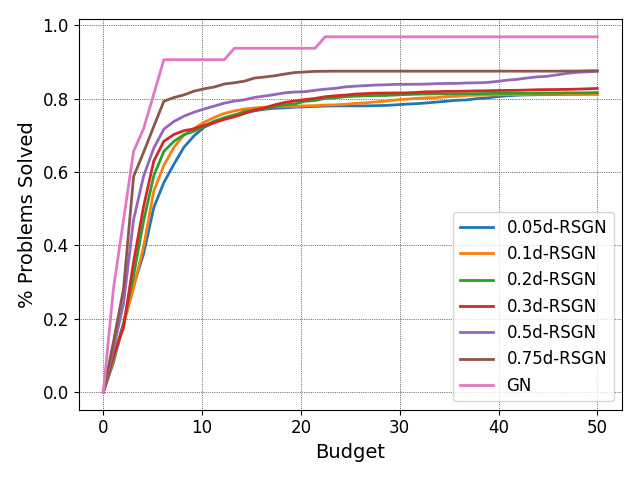}
\caption{RS-GN data profiles on the zero-residual test set with coordinate sampling for an accuracy of $\tau=10^{-1}$.}
\label{fig:cutest32_coordinate}
\end{minipage}
\begin{minipage}{.45\textwidth}
\centering
\includegraphics[width=\textwidth]{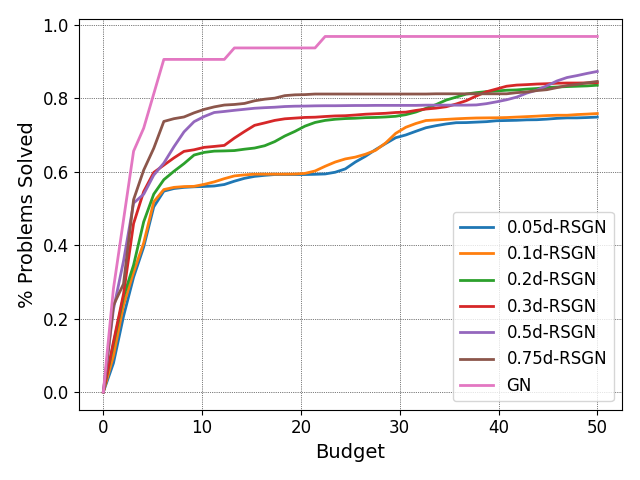}
\caption{RS-GN data profiles on the zero-residual test set with Gaussian sketching for an accuracy of $\tau=10^{-1}$.}
\label{fig:cutest32_gaussian}
\end{minipage}
\end{figure}

\begin{figure}[!p]
\centering
\begin{minipage}{.45\textwidth}
\centering
\includegraphics[width=\textwidth]{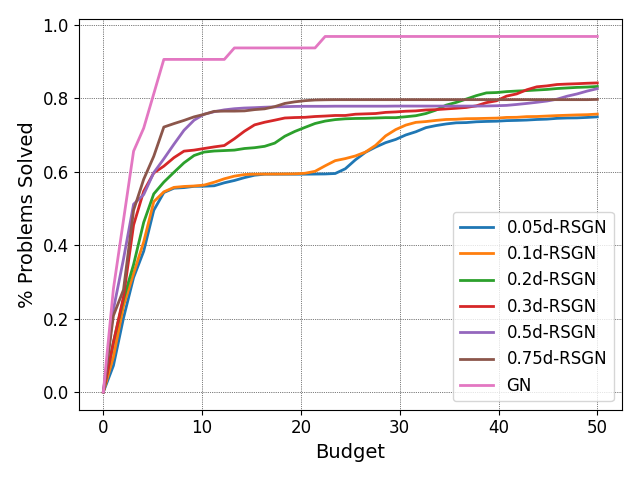}
\caption{RS-GN data profiles on the zero-residual test set with $3$-hashing sketching for an accuracy of $\tau=10^{-1}$.}
\label{fig:cutest32_hashing}
\end{minipage}
\begin{minipage}{.45\textwidth}
\centering
\includegraphics[width=\textwidth]{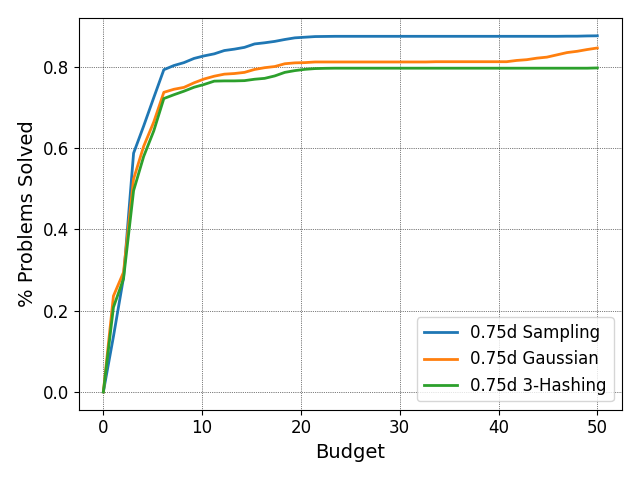}
\caption{RS-GN data profiles on the zero-residual test set with the various sketching matrices for an accuracy of $\tau=10^{-1}$.}
\label{fig:cutest32_across}
\end{minipage}
\end{figure}

\paragraph{Nonzero residual CUTEst Problems}
Next, we consider a test set of 19 nonzero-residual ($f(x^*)\neq 0$) nonlinear least squares problems from CUTEst, given in \autoref{tab:cutest19}. As before, these problems mostly have dimension around $d=100$ with a few having $d=59$ to $64$, for reasons previously discussed. We will also set the accuracy level $\tau=10^{-1}$ as before. 

\begin{table}[!h]
\centering
\begin{tabular}{lrrlrrlrr}  
\toprule
Name & $d$ & $n$ & Name & $d$ & $n$ & Name & $d$ & $n$ \\
\midrule
ARGLALE & 100 & 400 & FLOSP2HL & 59 & 59 & LUKSAN14 & 98 & 224 \\
ARGLBLE & 100 & 400 & FLOSP2HM & 59 & 59 & LUKSAN15 & 100 & 196 \\
ARWHDNE & 100 & 198 & FREURONE & 100 & 198 & LUKSAN16 & 100 & 196 \\
BRATU2DT & 64 & 64 & PENLT1NE & 100 & 101 & LUKSAN17 & 100 & 196 \\
CHEMRCTB & 100 & 100 & PENLT12NE & 100 & 200 & LUKSAN22 & 100 & 198 \\
DRCAVTY2 & 100 & 100 & LUKSAN12 & 98 & 192 & & & \\
FLOSP2HH & 59 & 59 & LUKSAN13 & 98 & 224 & & & \\
\bottomrule
\end{tabular}
\caption{The 19 CUTEst test problems in the nonzero-residual test set.}\label{tab:cutest19}
\end{table}

Figures \ref{fig:cutest19_coordinate}, \ref{fig:cutest19_gaussian} and \ref{fig:cutest19_hashing} show the performance of RS-GN with different choices of sketching matrices (sampling, Gaussian and $3$-hashing, respectively) and for different sizes of the subspace $l$ 
(as a fraction of $d$), in data profiles that measure cumulative Jacobian actions evaluations (meaning total number of Jacobian matrix-vector products that are used). Again, it is useful to compare across the different sketching matrices and to this end we plot the Jacobian actions data profile for a subspace size of $0.75d$ in \autoref{fig:cutest19_across}. The conclusions are similar to the zero-residual cases above. Thus our conclusions for these low-dimensional examples are that a good remit for the use of subspace methods is when the full problem information is not available or is computationally too expensive, and so the only feasible possibility is to query a subset of Jacobian actions at each iteration; then subspace methods are applicable and provide reasonable progress.

\begin{figure}[!p]
\centering
\begin{minipage}{.45\textwidth}
\centering
\includegraphics[width=\textwidth]{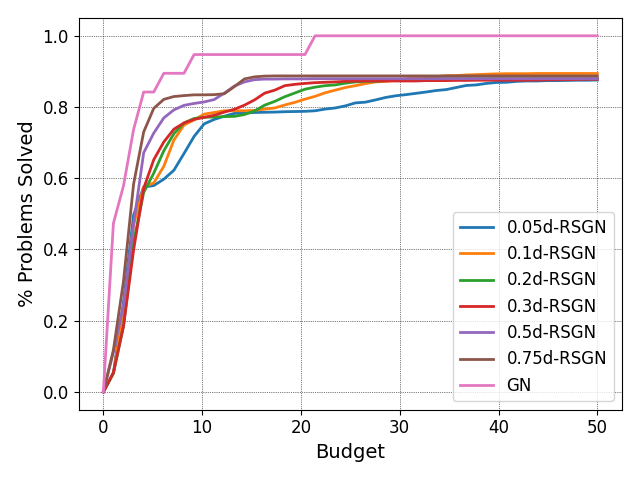}
\caption{RS-GN data profiles on the nonzero-residual test set with coordinate sampling for an accuracy of $\tau=10^{-1}$.}
\label{fig:cutest19_coordinate}
\end{minipage}
\begin{minipage}{.45\textwidth}
\centering
\includegraphics[width=\textwidth]{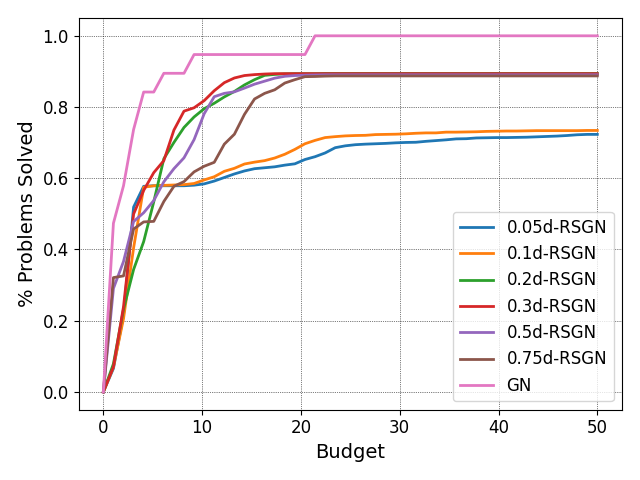}
\caption{RS-GN data profiles on the nonzero-residual test with Gaussian sketching for an accuracy of $\tau=10^{-1}$.}
\label{fig:cutest19_gaussian}
\end{minipage}
\end{figure}

\begin{figure}[!p]
\centering
\begin{minipage}{.45\textwidth}
\centering
\includegraphics[width=\textwidth]{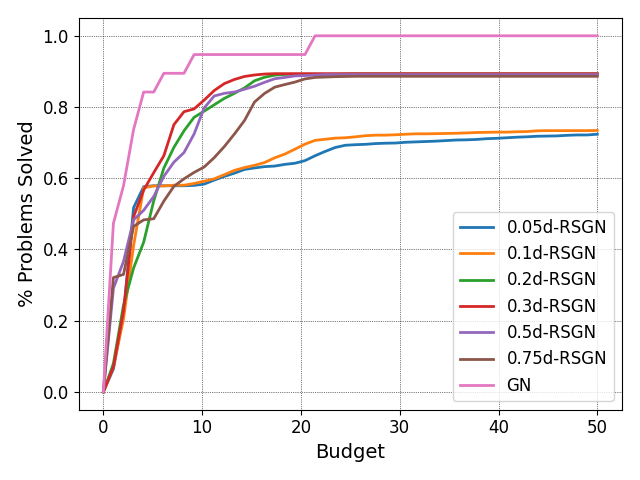}
\caption{RS-GN data profiles on the nonzero-residual test set with $3$-hashing sketching for an accuracy of $\tau=10^{-1}$.}
\label{fig:cutest19_hashing}
\end{minipage}
\begin{minipage}{.45\textwidth}
\centering
\includegraphics[width=\textwidth]{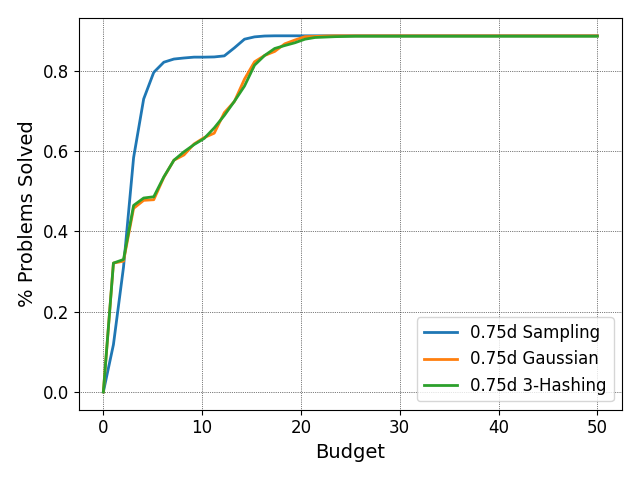}
\caption{RS-GN data profiles on the nonzero-residual test set with the various sketching matrices for an accuracy of $\tau=10^{-1}$.}
\label{fig:cutest19_across}
\end{minipage}
\end{figure}

\paragraph{Large-scale CUTEst Problems}
In this section we investigate the behaviour of RS-GN on three large(r)-scale ($d\approx5,000$ to $10,000$) problems from the CUTEst collection; see \autoref{tab:cutestind}. We run RS-GN five times on each problem until we achieve a $10^{-1}$ relative decrease in the objective, or failing that, for a maximum of 20 iterations. We plot the objective decrease against cumulative Jacobian action evaluations for each  run with block-sizes of 1\%, 5\%, 10\%, 50\%, 100\% of the original; see Figures \ref{fig:large_coordinate}, \ref{fig:large_gaussian} and 
\ref{fig:large_hashing}.

\begin{table}[!h]
\centering
\begin{tabular}{lrrlrrlrr}  
\toprule
Name & $d$ & $n$ & Name & $d$ & $n$ & Name & $d$ & $n$ \\
\midrule
ARTIF & 5,000 & 5,000 & BRATU2D & 4,900 & 4,900 & OSCIGRNE & 10,000 & 10,000 \\
\bottomrule
\end{tabular}
\caption{Three large-scale CUTEst test problems.}\label{tab:cutestind}
\end{table}

These figures show us that on ARTIF, block-coordinate variants of GN perform well but cannot surpass the efficient information use of the full-dimensional GN. However, RS-GN with Gaussian and hashing sketching can outperform GN when the given budget is low, which is often in practice. Similar behaviour occurs on OSCIRGNE with subspace methods outperforming GN initially, including for block coordinates variants.
On BRATU2D, all algorithms, including GN are slow initially, with GN then achieving a fast rate\footnote{We have carefully checked and subspace variants are not stagnating but progressing slowly on this problem.}. This fast rate can also be achieved by block variants with adaptive block size, as we illustrate next.

\begin{figure}[!h]
\begin{subfigure}{0.32\textwidth}
\includegraphics[width=\textwidth]{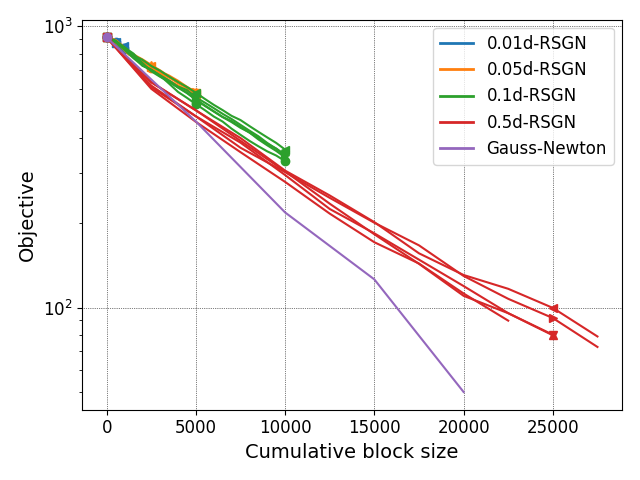}
\end{subfigure}
\begin{subfigure}{0.32\textwidth}
\includegraphics[width=\textwidth]{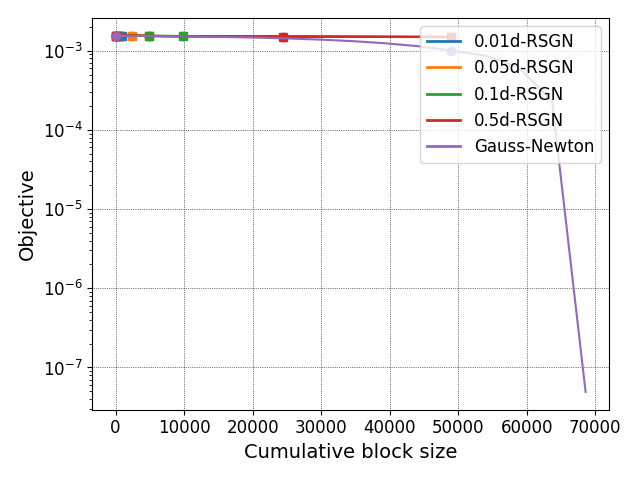}
\end{subfigure}
\begin{subfigure}{0.32\textwidth}
\includegraphics[width=\textwidth]{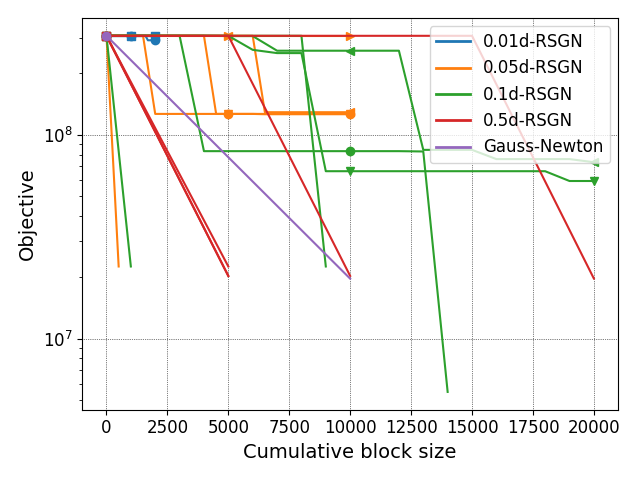}
\end{subfigure}
\caption{ARTIF (left), BRATU2D (middle) and OSCIGRNE (right) objective value against cumulative Jacobian action size for RS-GN with coordinate sampling.}
\label{fig:large_coordinate}
\end{figure}


\begin{figure}[!h]
\begin{subfigure}{0.32\textwidth}
\includegraphics[width=\textwidth]{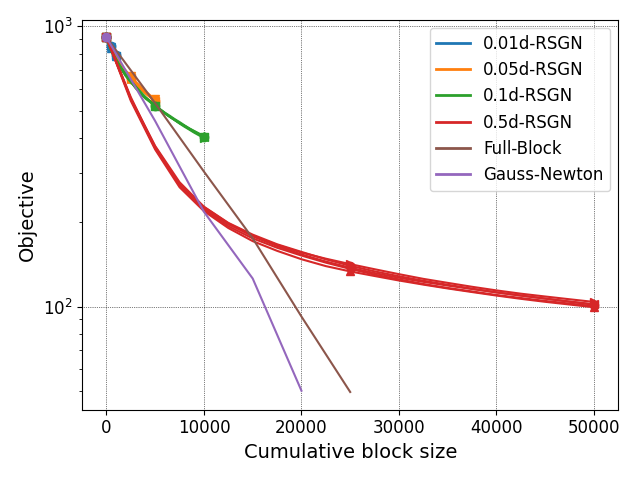}
\end{subfigure}
\begin{subfigure}{0.32\textwidth}
\includegraphics[width=\textwidth]{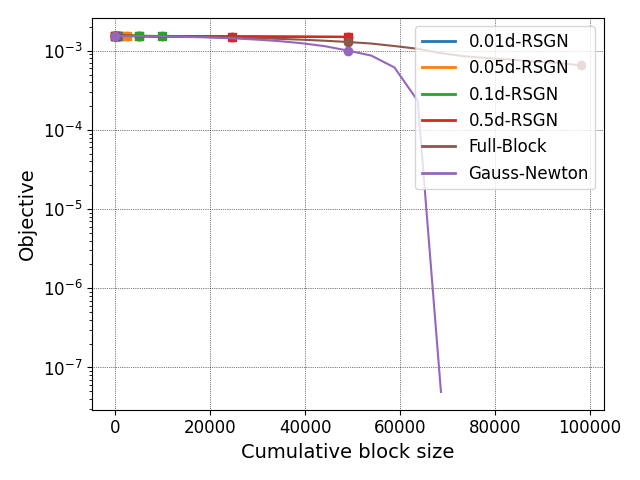}
\end{subfigure}
\begin{subfigure}{0.32\textwidth}
\includegraphics[width=\textwidth]{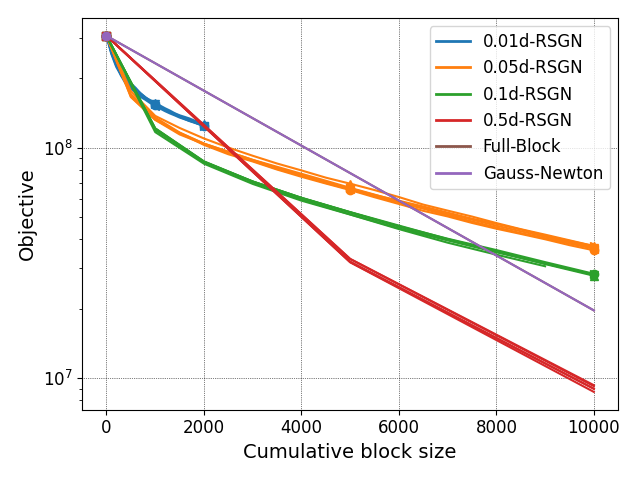}
\end{subfigure}
\caption{ARTIF (left), BRATU2D (middle) and OSCIGRNE (right) objective value against cumulative Jacobian action size for R-SGN with Gaussian sketching.}
\label{fig:large_gaussian}
\end{figure}


\begin{figure}[!h]
\begin{subfigure}{0.32\textwidth}
\includegraphics[width=\textwidth]{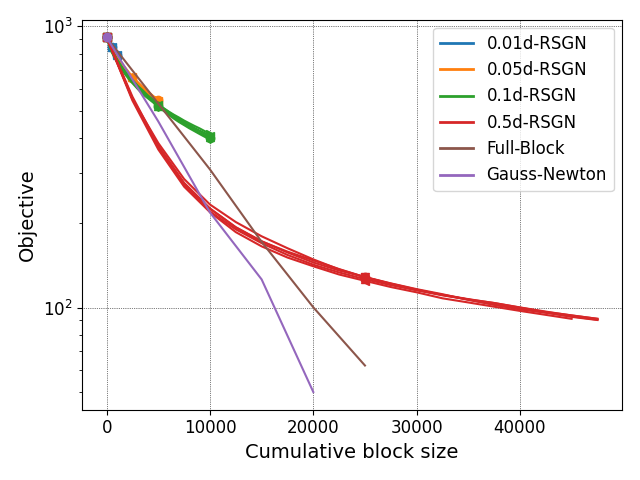}
\end{subfigure}
\begin{subfigure}{0.32\textwidth}
\includegraphics[width=\textwidth]{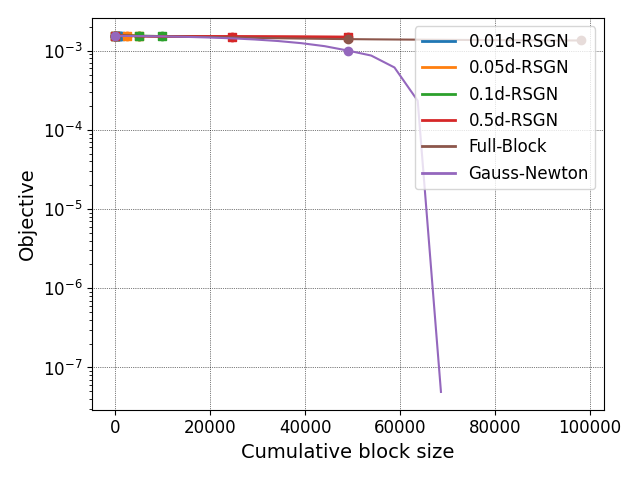}
\end{subfigure}
\begin{subfigure}{0.32\textwidth}
\includegraphics[width=\textwidth]{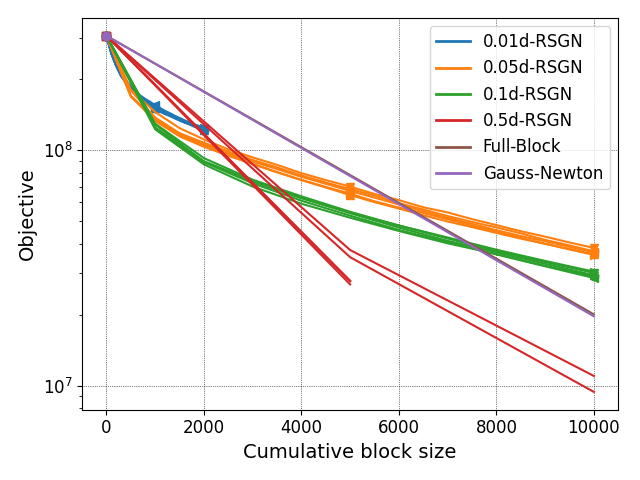}
\end{subfigure}
\caption{ARTIF (left), BRATU2D (middle) and OSCIGRNE (right) objective value against cumulative Jacobian action size for R-SGN with $3$-hashing sketching.}
\label{fig:large_hashing}
\end{figure}


\begin{figure}[!h]
\centering
\includegraphics[width=0.5\textwidth]{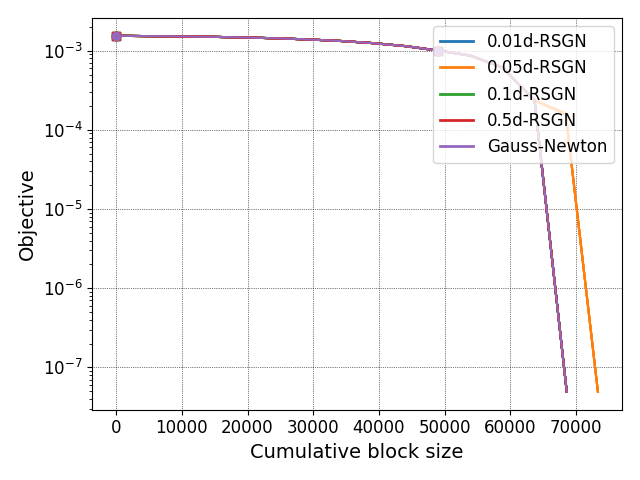}
\caption{BRATU2D objective value against cumulative Jacobian action size for Adaptive RS-GN with coordinate sampling starting at the subspace sizes indicated.}
\label{fig:bratu2d_adaptive}
\end{figure}

\paragraph{Adaptive RS-GN}
Finally, we highlight a variant of RS-GN where we adaptively increase the subspace dimension as the algorithm progresses. Starting from a fixed dimensional subspace, we adaptively increase the subspace by a fixed amount until we achieve a strong form of decrease in the reduced model, 
\[
\hat{m}_k(\sHat_k) \le \kappa \hat{m}_k(0)
\]
where $\sHat_k$ is the reduced trust region step and $\kappa\in(0,1)$ is our adaptivity parameter. In \autoref{fig:bratu2d_adaptive} we show adaptive R-SGN with coordinate sampling for the BRATU2D test problem, starting at the subspace sizes indicated in the figure and the increasing by increments of $500$.

\paragraph{Conclusions} Our preliminary numerical experiments validate our theoretical findings, and complement the simpler (convex)  
experiments on large scale logistic regression problems given in \cite{zhen:icml_BCGN}. 
As a way to improve the scalability of nonconvex optimization algorithms, our proposals here are scalable, in that the size of the subspace can be chosen fixed to a small value which reduces the linear algebra costs and the derivative actions calculations. 
We also note that block-coordinate Gauss-Newton methods have been applied successfully in applications, such as climate modelling  \cite{tett2017calibrating}. In fact, the motivation for our work here was very much inspired by the needs of this application, where full derivatives are incredibly expensive to compute (and even typical model-based derivative-free methods are too expensive to apply).

It is of course possible that other random matrix ensembles and associated scalings may further improve our theoretical and numerical results; though the fact that when $l=d$, we recover the full-dimensional first-order/trust-region or quadratic regularization complexity bounds  seems to imply that this may not be possible in general; but we expect it to be possible for special structured problems \cite{adilet1, adilet2, adilet3}. From a computational point of view, we are hopeful and encouraged by our current numerical results that further general improvements may still be achievable with careful and innovative random-subspace algorithm design, in an inspired combination with deterministic approaches.


\appendix
\section{Proof of the Main Result (\autoref{thm2})}
\label{BCGN:sec3}

The proof of \autoref{thm2} involves a technical analysis of the different types of iterations that can occur. An iteration can be true/false using \autoref{def::true_iters}, successful/unsuccessful (Step 3 of \autoref{alg:generic}) and with an $\alphaK$ above/below a certain value. The parameter $\alphaK$ is important due to \autoref{AA3} and \autoref{AA4} (that is, it influences the success of an iteration; and also the objective decrease in true and successful iterations).

Given that \autoref{alg:generic} runs for $N$ iterations, we use $N$ with different subscripts to denote the total number of the different types of iterations, detailed in \autoref{tab:it::count}.
We note that all iteration sets below are random variables because $\alphaK$, and whether an iteration is true/false, successful/unsuccessful, depend on the random model in Step 1 of \autoref{alg:generic} and the previous (random) steps.
\begin{table}[ht]
\begin{tabular}{ll}
Symbol            & Definition            \\                                    
$\nt$             & Number of true iterations                                         \\
$\nf$             & Number of false iterations                                                                \\
$\nts$            & Number of true and successful iterations                                                  \\
$\ns$             & Number of successful iterations                                                           \\
$\nuMe$             & Number of unsuccessful iterations                                                         \\
$\ntu$            & Number of true and unsuccessful iterations                                                \\
$\ntAlphaUpper$   & Number of true iterations such that $\alpha_k \leq \alphaLowOne$                             \\
$\nsAlphaUpper$   & Number of successful iterations such that $\alpha_k \leq \alphaLowOne$                       \\
$\ntAlphaLower$   & Number of true iterations such that $\alpha_k > \alphaLowOne$                             \\
$\ntsAlphaLower$  & Number of true and successful iterations such that $\alpha_k > \alphaLowOne$              \\
$\ntuAlphaLower$  & Number of true and unsuccessful iterations such that $\alpha_k > \alphaLowOne$            \\
$\nuAlphaLower$   & Number of unsuccessful iterations such that $\alpha_k > \alphaLowOne$                     \\
$\nsGammaAlpha$   & Number of successful iterations such that $\alpha_k > \gamma_1^{c}\alphaLowOne$           \\
$\ntsGammaCAlpha$ & Number of true and successful iterations such that $\alpha_k > \gamma_1^{c}\alphaLowOne$  \\
$\nfsGammaAlpha$  & Number of false and successful iterations such that $\alpha_k > \gamma_1^{c}\alphaLowOne$
\end{tabular}
\caption{List of random variables representing iteration counts given that \autoref{alg:generic} has run for $N$ iterations}
\label{tab:it::count}
\end{table}

The proof of \autoref{thm2} relies on the following three results relating the total number of different types of iterations. 

\paragraph{Relating the total number of true iterations to the total number of iterations}

\autoref{lem::Chernoff} shows that with high probability, a constant fraction of iterations of \autoref{alg:generic} are true. This result is a conditional variant of the Chernoff bound \cite{MR57518}.  

\begin{lemma}
\label{lem::Chernoff}
Let \autoref{AA2} hold with $\delta_S \in (0,1)$. Let \autoref{alg:generic} run for $N$ iterations. Then for any given $\delta_1 \in (0,1)$,
\begin{equation}
\P \left( \nt \leq (1-\delta_S)(1-\delta_1)N \right) \leq \chernoffLowerExponential, \label{Markov}
\end{equation}
where $N_T$ is defined in \autoref{tab:it::count}.
\end{lemma}

The proof of \autoref{lem::Chernoff} relies on the below technical result.

\begin{lemma}
    \label{lem::ChernoffHelper}
    Let \autoref{AA2} hold with $\delta_S \in (0,1)$. Let $T_k$ be defined in \eqref{eqn::t_k}.
    Then for any $\lambda>0$ and $N\in \N$, we have
        \begin{equation}
            \expectation{\eToMinusLambdaSumTk} \leq \exponentialMomentUpperTotal. \notag
        \end{equation}
\end{lemma}

\begin{proof}
Let $\lambda >0$. We use induction on $N$. For $N=1$, we want to show
        $\expectation{e^{-\lambda T_0}} \leq \exponentialMomentUpper$.
      Let $g(x) = e^{-\lambda x}$. Note that 
    \begin{equation}
        g(x) \leq g(0) + \left[ g(1) - g(0) \right]x, \text{ for any $x\in [0,1]$}, \label{eqn::convex}
    \end{equation}
    because $g(x)$ is convex. Substituting $x=T_0$, we have
        $\eToMinusLambda{T_0} \leq 1 + (\eToMinusLambdaMinusOne) T_0$.
       Passing to  expectation in the latter, we have that
    \begin{equation}
        \expectation{\eToMinusLambda{T_0}}  
        \leq 1 + (\eToMinusLambdaMinusOne)\expectation{T_0} . \label{eqn::inductionZeroUpper}
    \end{equation}
    Moreover, we have
       $ \expectation{T_0} \geq \probability{T_0=1} \geq 1-\delta_S$,
    where the first inequality is due to $T_0\geq 0$ and the second inequality, to \autoref{AA2}.
    Therefore, noting that $\eToMinusLambdaMinusOne <0$, \eqref{eqn::inductionZeroUpper} gives
    \begin{equation}
        \label{eqn::inductionFirstConclusion}
        \expectation{\eToMinusLambda{T_0}} 
        \leq 1 + (\eToMinusLambdaMinusOne) (1-\delta_S) \leq \exponentialMomentUpper, 
    \end{equation}
    where the last inequality comes from $1+y \leq e^y$ for $y \in \R$.
    
    Having completed the initial step for the induction, let us assume
    \begin{equation}
        \label{eqn::inductionAssumption}
        \expectation{  \eToMinusLambdaSumTkNMinusTwo } \leq \left[ \exponentialMomentUpper\right]^{N-1}.
    \end{equation}
    We have 
    \begin{align}
        & \expectation{\eToMinusLambdaSumTk} \notag 
        = \expectation{\conditionalE{\eToMinusLambdaSumTk}{\SZeroDotsToXNMinusOne}}\notag \\
        &= \expectation{\eToMinusLambda{\sumTkToNMinusTwo} \conditionalE{\eToMinusLambda{T_{N-1}}}{\SZeroDotsToXNMinusOne}} 
        = \expectation{ \eToMinusLambda{ \sumTkToNMinusTwo } \conditionalE{ \eToMinusLambda{ T_{N-1} }}{ x_{N-1}}}, \label{eqn::inductionTotal}
    \end{align}
    where the first equality is due to the Tower property and the last equality follows from $T_{N-1}$ being conditionally independent of the past iterations $T_0, T_1, \dots, T_{N-2}$ given $x_{N-1}$ (see \autoref{AA2}).
    Substituting $x=T_{N-1}$ in \eqref{eqn::convex}, and taking conditional expectations, we have that
    \begin{equation}
        \conditionalE{\eToMinusLambda{T_{N-1}}}{x_{N-1}}
        \leq 1 + (\eToMinusLambdaMinusOne) \conditionalE{T_{N-1}}{x_{N-1}}.\notag
    \end{equation}
    On the other hand, we have that $\conditionalE{T_{N-1}}{x_{N-1}} \geq \conditionalP{T_{N-1}=1}{x_{N-1}} \geq 1-\delta_S$, where we used $T_{N-1} \geq 0$ to derive the first inequality and $\conditionalP{T_{N-1}=1}{x_{N-1} = \bar{x}_{N-1}} \geq 1 - \delta$ for any $\bar{x}_{N-1}$ (see \autoref{AA2}) to derive the second inequality. 
    Hence, we obtain the corresponding relation to \eqref{eqn::inductionFirstConclusion}, namely,
    \begin{equation}
        \conditionalE{\eToMinusLambda{T_{N-1}}}{x_{N-1}}\leq \exponentialMomentUpper.
    \end{equation}
        It then follows from \eqref{eqn::inductionTotal} that
    \begin{equation}
        \expectation{\eToMinusLambdaSumTk} 
        \leq \exponentialMomentUpper \expectation{ \eToMinusLambda{ \sumTkToNMinusTwo }} 
        \leq \exponentialMomentUpperTotal, \notag
    \end{equation}
    where we used \eqref{eqn::inductionAssumption} to obtain the last inequality.
\end{proof}

\begin{proof}[Proof of \autoref{lem::Chernoff}]
Note that since $N$ is the total number of iterations, we have $N_T = \sumTk$, where $T_k$ is defined in \eqref{eqn::t_k}. Applying Markov inequality, we have that for any $\lambda >0$,
\begin{align}
    &\probability{N_T \leq (1-\delta_S)(1-\delta_1) N}
    = \probability{\eToMinusLambda{N_T} \geq \eToMinusLambda{ (1-\delta_S) (1-\delta_1) N }}  \notag \\
    & \leq \expectation{\eToMinusLambda{N_T}} e^{\lambda(1-\delta_S) (1-\delta_1) N}
    = \expectation{ \eToMinusLambdaSumTk} e^{\lambda(1-\delta_S) (1-\delta_1) N}
    \leq e^{ N(\eToMinusLambdaMinusOne)(1-\delta_S)+\lambda(1-\delta_S) (1-\delta_1) N }, \label{eqn::Chernoff_proof1}
\end{align}
where we used Lemma \ref{lem::ChernoffHelper} to derive the last inequality. 
Choosing $\lambda = -\log(1-\delta_1) >0$, we have from \eqref{eqn::Chernoff_proof1}
$$
    \probability{N_T \leq (1-\delta_S)(1-\delta_1) N } 
    \leq e^{N (1-\delta_S) \left[ \deltaOneComplicated \right]} 
    \leq \chernoffLowerExponential, \notag
$$
where we used $\deltaOneComplicated \leq -\delta_1^2/2$ for $\delta_1 \in (0,1)$. 
\end{proof}



\paragraph{Relating the total number of true iterations with $\alphaK \leq \alphaMin$ to the total number of iterations}

The next Lemma shows that we can have at most a constant fraction of iterations of \autoref{alg:generic} that are true with $\alphaK \leq \alphaMin$.
\begin{lemma} \label{lm::Gratton}
    Let \autoref{AA3} hold with $\alphaLow>0$ and $c\in \N^+$ and let $\alphaMin$ associated with $\alphaLow$ be defined in \eqref{eqn::alphaMin} with $\newL \in \N^+$. Let $\epsilon>0$, $N \in \N$ be the total number of iterations; and $\ntAlphaUpper$ be defined in \autoref{tab:it::count}. Suppose $N\leq \nEps$. Then 
	\begin{equation}
	    \ntAlphaUpper \leq \frac{N}{c+1}. \label{eqn::tmp8}
	\end{equation}
\end{lemma}

\begin{proof}
Let $k\leq N-1$ since the total number of iterations taken by \autoref{alg:generic}  is assumed to be $N$.
It follows from $N\leq \nEps$ that $k < \nEps$ and by definition of
$\alphaMin$ (\autoref{lem::alphaMin}), iteration $k$ being true with $\alpha_k \leq \alphaMin$ implies that iteration $k$ is successful (with $\alpha_k \leq \alphaMin$). 
Therefore we have 
\begin{equation}
    \ntAlphaUpper \leq \nsAlphaUpper. \label{eqn::tmp6}
\end{equation}
If $\nsAlphaUpper =0$, then $\ntAlphaUpper=0$ and \eqref{eqn::tmp8} holds. Otherwise let 
\begin{equation}
    \kBar = \max \set{k\leq N-1: \text{iteration $k$ is successful and $\alpha_k \leq \alphaLowOne$}}.\label{eqn::tmp5}
\end{equation}
Then for each $k \in \set{0, 1, \dots, \kBar}$, we have that either iteration $k$ is successful and $\alpha_k \leq \alphaLowOne$, in which case $\alpha_{k+1} = \gamma_2 \alpha_k$ (note that \eqref{eqn::alphaMinUpperByGammaTwoOverAlphaZero} and $\alphaK\leq \alphaMin$ ensure $\max \set{\gammaTwo \alphaK, \alphaMax} = \gammaTwo\alphaK$); or otherwise $\alpha_{k+1} \geq \gamma_1 \alpha_k$ (which is true for any iteration of \autoref{alg:generic}). Hence after $\kBar+1$ iterations, we have
\begin{align}
    \alpha_{\kBar+1} \geq \alpha_0 \gamma_2^{\nsAlphaUpper} \gamma_1^{\kBar+1 -\nsAlphaUpper} 
    = \alpha_0 \left( \frac{\gamma_2}{\gamma_1}\right)^{\nsAlphaUpper}\gamma_1^{\kBar+1}
    \geq \alpha_0 \left( \frac{\gamma_2}{\gamma_1}\right)^{\nsAlphaUpper}\gamma_1^{N}, \label{eqn:tmp29}
\end{align}
where we used $\kBar+1\leq N$ in the last inequality. 
On the other hand, we have 
  $\alpha_{\kBar+1} = \gamma_2 \alpha_{\kBar} \leq \gamma_2 \alphaLowOne$, 
due to the fact that iteration $\kBar$ is successful, $\alpha_{\kBar}\leq \alphaLowOne$ and \eqref{eqn::tmp5}.
Therefore, combining these with \eqref{eqn:tmp29}, we have $\gamma_2\alphaLowOne \geq \alpha_{\kBar+1} \geq \alpha_0 \left( \frac{\gamma_2}{\gamma_1}\right)^{\nsAlphaUpper}\gamma_1^{N}$. Taking logarithm on both sides, we have
\begin{equation}
    \log(\gamma_2 \alphaLowOne) \geq \log(\alpha_0) + \nsAlphaUpper\log(\frac{\gamma_2}{\gamma_1}) + N \log(\gamma_1),\notag
\end{equation}
which rearranged, gives
\begin{equation}
    \nsAlphaUpper \leq p_0 N + p_1, \notag
\end{equation}
with $p_0 = \frac{\logOneOverGammaOne}{\logGammaTwoOverGammaOne} = \frac{1}{c+1}$ and $p_1 = \pOneDef = \frac{c-\newL}{c+1} \leq 0$ as $\newL \geq c>0$. Therefore we have $\nsAlphaUpper \leq \frac{N}{c+1}$ and \eqref{eqn::tmp6} then gives the desired result. 
\end{proof}


\paragraph{Relating the number of unsuccessful and successful iterations}

The next Lemma extends to the case of random models a common result for deterministic and adaptive nonlinear optimization algorithms. It formalises the intuition that one cannot have too many unsuccessful iterations with $\alphaK > \alphaMin$ compared to successful iterations with $\alphaK > \gammaOneC \alphaMin$, since unsuccessful iterations reduce $\alphaK$ and only successful iterations with $\alphaK > \gammaOneC \alphaMin$ may compensate for these decreases. The conditions that $\alphaMin = \alphaZero\gammaOne^\newL$, $\gammaTwo = \frac{1}{\gammaOneC}$ and $\alphaMax = \alphaZero \gammaOne^p$ for some $\newL, c, p \in \N^+$ are crucial in the (technical) proof. 

\begin{lemma}\label{lm::Katya}
	Let \autoref{AA3} hold with $\alphaLow>0$. Let $\alphaMin$ associated with $\alphaLow$ be defined in \eqref{eqn::alphaMin} with $\newL\in \N^+$. Let $N \in \N$ be the total number of iterations of \autoref{alg:generic} and $\nuAlphaLower$, $\nsGammaAlpha$ be defined in \autoref{tab:it::count}. Then
	\begin{equation}
	    \nuAlphaLower \leq \newL + c\nsGammaAlpha. \notag
	\end{equation}
\end{lemma}

\begin{proof}
    Define 
    \begin{equation}
    \betaK = \logBaseGammaOne{\frac{\alphaK}{\alphaZero}}. \label{eqn::betaKDef}
    \end{equation}
    Note that since $\alpha_{k+1} = \gamma_1 \alpha_k$ if iteration $k$ is successful and $\alpha_{k+1} = \minMe{\alphaMax}{\gamma_2 \alphaK}$ otherwise, $\gamma_2 = \gammaTwoExpression$ and $\alpha_{max} = \alpha_0 \gamma_1^p$ with $c,p \in \N^+$, we have that $\betaK \in \Z$. Moreover, we have that $\alphaK=\alphaZero$ corresponds to $\betaK = 0$, $\alphaK=\alphaMin$ corresponds to $\betaK = \newL$ and $\alphaK = \gamma^c \alphaMin$ corresponds to $\betaK= \newL+c$. Note also that on successful iterations, we have $\alphaKPlusOne \leq \gammaTwo \alphaK = \gammaOne^{-c} \alphaK$ (as $\alphaKPlusOne = \minMe{\alphaMax}{\gammaTwo\alphaK}$) so that $\betaKPlusOne \geq \betaK -c$; and on unsuccessful iterations, we have $\betaKPlusOne = \betaK +1$.
	
	Let $\kStartOne=-1$; and define the following sets.
	\begin{align}
	    &\aOne = \set{k \in  \openClosedInterval{\kStartOne}{N-1}\intersect \N: \betaK=\newL}. \label{eqn::tmp11} \\
	    & \kEndOne = \twoCases{\inf \aOne}{\text{if $\aOne \neq \emptyset$}}{N}{\text{otherwise}}  \notag \\
	    & \mOneOne = \set{k \in \openInterval{\kStartOne}{\kEndOne}:  \text{iteration $k$ is unsuccessful with $\betaK < \newL$}} \notag \\
	    & \mTwoOne = \set{k \in \openInterval{\kStartOne}{\kEndOne}:  \text{iteration $k$ is successful with $\betaK < \newL+c $ }} \label{eqn::tmp12}.
	\end{align}
	Let $\nOneOne = |\mOneOne|$ and $\nTwoOne=| \mTwoOne|$, where $|.|$ denotes the cardinality of a set. 
	
	If $\kEndOne <N$, we have that $\kEndOne$ is the first time $\betaK$ reaches $\newL$. Because $\betaK$ starts at $0<\newL$ when $k=0$; $\betaK$ increases by one on unsuccessful iterations and decreases by an integer on successful iterations (so that $\betaK$ remains an integer). So for $\kInStartEndOne$, all iterates have $\betaK < \newL < \newL+c$. It follows then the number of successful/unsuccessful iterations for $\kInStartEndOne$ are precisely $\nOneOne$ and $\nOneTwo$ respectively. Because $\betaK$ decreases by at most $c$ on successful iterations, increases by one on unsuccessful iterations, starts at zero and $\beta_{\kEndOne}\leq \newL$, we have $0 + \nOneOne - c\nTwoOne \leq \newL$ (using $\beta_{\kEndOne} \geq \beta_{\kStartI+1} + \nOneOne - c\nTwoOne$). Rearranging gives
	\begin{equation}
	    \nOneOne \leq c \nTwoOne + \newL. \label{eqn::tmp3}
	\end{equation}
	
	If $\kEndOne = N$, then we have that $\betaK <\newL$ for all $k \leq N-1$ and so $\beta_{\kEndOne} \leq \newL$. In this case we can derive \eqref{eqn::tmp3} using the same argument. Moreover, since $\kEndOne=N$, we have that 
	\begin{align}
	    \nOneOne &= \nuAlphaLower, \label{eqn:tmp30} \\
	    \nOneTwo &= \nsGammaAlpha. \label{eqn:tmp31}
	\end{align}
	The desired result then follows. 
	
	Hence we only need to continue in the case where $\kEndOne < N$, in which case
	let
	\begin{align}
	    \bOne & = \set{k \in \closedInterval{\kEndOne}{N-1}: \text{iteration $k$ is successful with $\betaK < \newL+c $ } } \notag \\
	    \kStartTwo &= \twoCases{\inf \bOne}{\text{if $\bOne \neq \emptyset$}}{N}{\text{otherwise}}. \notag
	\end{align}

	Note that there is no contribution to $\nsGammaAlpha$ or $\nuAlphaLower$ for $k \in \closedOpenInterval{\kEndOne}{\kStartTwo}$. 
	There is no contribution to $\nsGammaAlpha$ because $\kStartTwo$ is the first iteration (if any) that would make this contribution. 
	Moreover, since $\beta_{\kEndOne}=\newL$ by definition of $\kEndOne$, the first iteration with $\betaK<\newL$ for $k\geq \kEndOne$ must be proceeded by a successful iteration with $\betaK< \newL+c$ (note that in particular, since $\kStartTwo$ is the first such iteration, we have $\beta_{\kStartTwo}\geq \newL$). 
	Therefore there is no contribution to $\nuAlphaLower$ either for $k \in \closedOpenInterval{\kEndOne}{\kStartTwo}$. Hence if $\kStartTwo = N$, we have \eqref{eqn:tmp30}, \eqref{eqn:tmp31} and \eqref{eqn::tmp3} gives the desired result. 
	
	Otherwise similarly to \eqref{eqn::tmp11}--\eqref{eqn::tmp12},
	let 
	\begin{align}
	    &\aTwo = \set{k \in \openClosedInterval{\kStartTwo}{N-1} \intersect \N: \betaK=\newL}. \notag \\
	    & \kEndTwo = \twoCases{\inf \aTwo }{\text{if $\aTwo \neq \emptyset$}}{N}{\text{otherwise}} \notag \\
	    & \mOneTwo = \set{k \in \openInterval{\kStartTwo}{\kEndTwo}: \text{iteration $k$ is unsuccessful with $\betaK < \newL$}} \notag \\
	    & \mTwoTwo = \set{k \in \openInterval{\kStartTwo}{\kEndTwo}: \text{iteration $k$ is successful with $\betaK < \newL+c $ }}. \notag
	\end{align}	
	And let $\nOneTwo = |\mOneTwo|$ and $\nTwoTwo=| \mTwoTwo|$. Note that for $k\in \openInterval{\kStartTwo}{\kEndTwo}$, we have $\newL-c \leq \beta_{\kStartTwo+1}$ and $\beta_{\kEndTwo} \leq \newL$ (the former is true as $\beta_{\kStartTwo\geq l}$ and iteration $\kStartTwo$ is successful). Therefore we have
	\begin{equation}
	    \newL-c + \nOneTwo-c\nTwoTwo 
	    \leq \beta_{\kStartTwo+1} + \nOneTwo - c\nTwoTwo \leq \beta_{\kEndTwo} \leq \newL.  \notag
	\end{equation}
	Rearranging gives
	\begin{equation}
	    \nOneTwo \leq c \nTwoTwo+ \newL - [\newL-c] = c\nTwoTwo +c, \label{eqn::tmp4}
	\end{equation}
	
	Let $\hatNOneOne$ be the total number of iterations contributing to $\nuAlphaLower$ with $k \in \closedInterval{\kEndOne}{\kStartTwo}$; and $\hatNOneTwo$ be the total number of iterations contributing to $\nsGammaAlpha$ with $k \in \closedInterval{\kEndOne}{\kStartTwo}$. Since there is no contribution to either for $k\in \closedOpenInterval{\kEndOne}{\kStartTwo}$ as argued before, and iteration $\kStartTwo$ by definition contributes to $\nsGammaAlpha$ by one, we have
	\begin{align}
	    \hatNOneOne = 0, \label{eqn::tmp13}\\
	    \hatNOneTwo = 1.\label{eqn::tmp14}
	\end{align}
	
	Using \eqref{eqn::tmp3}, \eqref{eqn::tmp4}, \eqref{eqn::tmp13} and \eqref{eqn::tmp14},we have
	\begin{equation}
	    \nOneOne+ \hatNOneOne + \nOneTwo \leq c \left( \nTwoOne + \hatNOneTwo + \nTwoTwo \right) +\newL. \label{eqn::tmp15}
	\end{equation}
	If $\kEndTwo = N$ the desired result follows. 
	Otherwise define $\bTwo$ in terms of $\kEndTwo$, 
	and $\kStartThree$ in terms of $\bTwo$ similarly as before. 
	If $\kStartThree=N$, 
	then we have the desired result as before. 
	Otherwise repeat what we have done (define $A^{(3)}$,
	$k_{end}^{(3)}$, $M_1^{(3)}$, $M_2^{(3)}$ etc). 
	Note that we will reach either $\kEndI = N$ 
	for some $i\in \N$ 
	or $\kStartI = N$ for some $i \in \N$, 
	because if $\kEndI<N$ and $\kStartI <N$ for all $i$, 
	we have that $\kStartI < \kEndI \leq \kStartIPlusOne$ 
	by definitions. So $\kStartI$ is strictly increasing,
	contradicting $\kStartI <N$ for all $i$. 
	In the case wither $\kEndI=N$ or $\kStartI=N$, 
	the desired result will follow using our previous argument.
\end{proof}

\paragraph{An intermediate result bounding the total number of iterations}
Using \autoref{lem::Chernoff}, \autoref{lm::Gratton} and \autoref{lm::Katya}, we prove an upper bound on the total number of iterations of \autoref{alg:generic} in terms of the number of true and successful iterations when $\alphaK$ is sufficiently large.

\begin{lemma}\label{lem:AssOneTwo}
    Let \autoref{AA2} and \autoref{AA3} hold with 
    $ \deltaS \in (0,1)$, $c, \newL \in \N^+$. 
    Let $N$ be the total number of iterations. 
    Then for any $\deltaOne \in (0,1)$ such that
    $\gDeltaSDeltaOne >0 $,   where $\gDeltaSDeltaOne$ is defined in \eqref{eqn:gDeltaSDeltaOneDef}, we have that 
    $$
        \probability{N < \gDeltaSDeltaOne \squareBracket{
         \ntsAlphaZeroGammaOneCL
         + \frac{\newL}{1+c}}} \geq 
         1- \chernoffLowerExponential.
    $$
\end{lemma}

\begin{proof}
We decompose the number of true iterations as 
\begin{equation}
    \nt = \ntAlphaUpper + \ntAlphaLower = \ntAlphaUpper + \ntsAlphaLower + \ntuAlphaLower \leq \ntAlphaUpper + \ntsAlphaLower + \nuAlphaLower, \label{eqn::tmp7}
\end{equation}
where $\nt, \ntAlphaUpper, \ntAlphaLower, \ntsAlphaLower, \ntuAlphaLower, \nuAlphaLower$ are defined in \autoref{tab:it::count}.
From \autoref{lm::Katya}, we have
\begin{align}
    \nuAlphaLower & \leq \tCOne + \tCTwo \nsGammaAlpha 
     = \tCOne + \tCTwo\ntsGammaCAlpha + \tCTwo\nfsGammaAlpha \notag \\
    & \leq \tCOne +\tCTwo \ntsGammaCAlpha + \tCTwo\nf 
     \leq \tCOne + \tCTwo \ntsGammaCAlpha + \tCTwo(N-\nt). \notag
\end{align}
It then follows from \eqref{eqn::tmp7} that
    $\nt \leq \ntAlphaUpper + \ntsAlphaLower + \tCOne + \tCTwo \ntsGammaCAlpha + \tCTwo(N-\nt)$.
Rearranging, we have
\begin{equation}
    \nt \leq \frac{\ntAlphaUpper}{\onePlusTCTwo}+ \frac{1}{\onePlusTCTwo}\squareBracket{\ntsAlphaLower + \tCTwo\ntsGammaCAlpha} + \frac{\tCOne+\tCTwo N}{\onePlusTCTwo}. \notag
\end{equation}
Using \autoref{lm::Gratton} to bound $\ntAlphaUpper$, $\ntsAlphaLower \leq \ntsGammaCAlpha$, and $ \alphaMin = \alphaZero \gammaOne^\newL$ gives 
\begin{equation}
    N_T \leq 
    \squareBracket{1 - \frac{c}{(c+1)^2}}N
    + N_{TS, \underline{\alphaZero \gammaOne^{c+\newL}}}
    + \frac{\newL}{1+c}, \label{tmp:2021-12-31-2}
\end{equation}
which combined with \autoref{lem::Chernoff} and 
rearranged, gives the result. 
\end{proof}



\paragraph{The bound on true and successful iterations}

The next lemma bounds the total number of true and successful iterations with $\alphaK > \alphaZero \gammaOne^{c+\newL}$.

\begin{lemma} \label{lm::bound_T_S_with_artificial_alpha_low}
Let \autoref{AA4} and \autoref{AA5} hold.
Let $\epsilon>0$ and $N \leq \nEps$ be the total number of iterations.
 Then 
\begin{align}
 N_{TS, \underline{\alphaZero \gammaOne^{c+\newL}}} \leq \frac{f(x_0) - f^*}{
h(\epsilon, \alphaZero \gammaOne^{c+\newL})} \notag
\end{align}
where $f^*$ is defined in \eqref{eqn::fStar}, and $x_0$ is chosen at the start of \autoref{alg:generic}.
\end{lemma}

\begin{proof}
Using \autoref{AA5} and \autoref{AA4} respectively for the two inequalities below, we have
\begin{align}
    f(x_0) - f(x_{N})
    &= \sum_{k=0}^{N-1} f(\xK) - f(\xKPlusOne) \notag \\
    &\geq \sum_{\IterKTrueandSuccssfulWithAlphaKGeqAlphaMin} f(x_k) - f(\xKPlusOne)  \notag \\
    &\geq \sum_{\IterKTrueandSuccssfulWithAlphaKGeqAlphaMin} h(\epsilon, \alphaZero \gammaOne^{c+\newL}) \notag \\
    &= N_{TS, \underline{\alphaZero \gammaOne^{c+\newL}}} h(\epsilon, \alphaZero \gammaOne^{c+\newL}).\label{eqn::tmp1}
\end{align}
Noting that $f(x_{N}) \geq f^*$ and $h(\epsilon, \alphaZero \gammaOne^{c+\newL})>0$ 
by \autoref{AA4}, and rearranging \eqref{eqn::tmp1}, gives the required result.
\end{proof}

\paragraph{Proving the main result}

We are ready to prove \autoref{thm2} using \autoref{lem:AssOneTwo} and \autoref{lm::bound_T_S_with_artificial_alpha_low}.

\begin{proof}[Proof of \autoref{thm2}]
As $ \nEps \geq N $, \autoref{lm::bound_T_S_with_artificial_alpha_low} implies that
$$ \fZeroMinusfStarOverH
    \geq \ntsAlphaZeroGammaOneCL.$$
    This and  \eqref{eqn::n_upper_2} imply
   $$ N \geq \gDeltaSDeltaOne \squareBracket{
         \ntsAlphaZeroGammaOneCL
         + \frac{\newL}{1+c}}.$$
Therefore by \autoref{lem:AssOneTwo}, we have
$\probability{\nEps \geq N} \leq 
\probability{N \geq \gDeltaSDeltaOne \squareBracket{
         \ntsAlphaZeroGammaOneCL
         + \frac{\newL}{1+c}}}
         \leq \chernoffLowerExponential$.
\end{proof}

\section{Proof of some useful lemmas}
\label{BCGN:aux-app}

\paragraph{Proof of Lemma \ref{lem:GaussJLEmbedding}}
\begin{proof}
Since \eqref{eqn:tmp36} is invariant to the scaling of $y$ and is trivial for $y=0$, 
we may assume without loss of generality that $\normTwo{y}=1$. 
Let $R = \sqrt{l}S$, so that each entry of $R$ is distributed independently as $N(0,1)$. 
Then because the sum of independent Gaussian random variables is distributed as a Gaussian random variable; $\normTwo{y}=1$; and the fact that rows of $S$ are independent; we have that the entries of $Ry$, denoted by $z_i$ for $i\in [l]$, are independent $N(0,1)$ random variables. Therefore, for any $-\infty < q < \frac{1}{2}$, we have that
\begin{equation}
    \expectation{e^{q \normTwo{Ry}^2}} = \expectation{e^{q \sum_{i=1}^l z_i^2}} = \prod_{i=1}^l \expectation{e^{qz_i^2}} = (1-2q)^{-l/2}, \label{eqn:tmp21}
\end{equation}
where we used $\expectation{e^{qz_i^2}} = \frac{1}{1-2q}$ for $z_i \in N(0,1)$ and $-\infty < q < \frac{1}{2}$.
Hence, by Markov inequality, we have that, for $q <0$, 
\begin{equation}
    \probability{\normTwo{Ry}^2 \leq l (1-\epS)} = \probability{e^{q \normTwo{Ry}^2}\geq e^{ql(1-\epS)}} 
    \leq \frac{\expectation{\eToQRySqaured}}{\eToQLOneMinusEps} = (1-2q)^{-l/2} \eToMinusQLOneMinusEps, \label{eqn:tmp23}
\end{equation}
where the last inequality comes from \eqref{eqn:tmp21}.
Noting that 
\begin{equation}
    (1-2q)^{-l/2} \eToMinusQLOneMinusEps = \exp \squareBracket{-l \bracket{\frac{1}{2} \log(1-2q)+ q(1-\epS)}}, \label{eqn:tmp22}
\end{equation}
which is minimised at $q_0 = -\frac{\epS}{2(1-\epS)}<0$, we choose $q=q_0$ and 
the right hand side of \eqref{eqn:tmp23} becomes 
\begin{equation}
    e^{\frac{1}{2}l\squareBracket{\epS + \log(1-\epS)}} \leq e^{-\frac{1}{4}l\epS^2}, \label{eqn:tmp24}
\end{equation}
where we used $\log(1-x) \leq -x -x^2/2$, valid for all $x \in [0,1)$.
Hence we deduce
\begin{align}
      &\probability{\normTwo{Sy}^2 
    \leq (1-\epS) \normTwo{y}^2}
   = \probability{\normTwo{Sy}^2 \leq (1-\epS)} \texteq{by $\normTwo{y}=1$} \notag \\
    &= \probability{\normTwo{Ry}^2 \leq l(1-\epS)} \texteq{by $S = \frac{1}{\sqrt{l}}R$} \notag \\
    & \leq e^{-\frac{l\epS^2}{4}} \texteq{by \eqref{eqn:tmp24} and \eqref{eqn:tmp23}}. \notag
\end{align}
\end{proof}

\paragraph{Proof of Lemma \ref{lem:Taylor}}
\begin{proof}
The $L$-Lipschitz continuity properties of the gradient imply that
\begin{equation}
    \abs{f(x_k + S_k \sKHat) - \innerProduct{S_k\gradFK}{\sKHat}}
    \leq \frac{L}{2}\normTwo{S_k^T \sKHat}^2.
\end{equation}
The above equation and triangle inequality provide
\begin{align}
     \abs{f(x_k + s_k) - \mkHatSkHat} 
     = & \abs{f(x_k + s_k) - f(x_k) - \innerProduct{\sKGradFK}{\sKHat}
    - \frac{1}{2}\innerProduct{\SKTransposedsKHat}{B_k \SKTransposedsKHat}}
    \notag\\
    & \leq \bracket{\frac{L}{2} + \frac{1}{2}\normTwo{B_k}}
        \normTwo{S_k^T \sKHat}^2 
    \leq \frac{L + \BMax}{2} \normTwo{S_k^2 \sKHat},
\end{align}
where we used $\normTwo{B_k} \leq \BMax$ to derive the last inequality.
\end{proof}

{\small{\bibliography{2021_04_07_proper_ref.bib}}}
\bibliographystyle{abbrv}

\end{document}